\documentclass[a4paper,12pt]{article}

\usepackage[english]{babel} 
\usepackage{amsmath,amsthm,amsfonts,amssymb,graphicx,units}

\theoremstyle{plain}
\newtheorem{theo}{Theorem}[section]
\newtheorem{lem}[theo]{Lemma}

\newtheorem{cor}[theo]{Corollary}
\newtheorem{rem}[theo]{Remark}

\theoremstyle{definition}
\newtheorem{ex}{Example}
\numberwithin{equation}{section}

\newenvironment{dedi}[1]{\vspace*{#1mm}\center\sf}{\endcenter}
\newenvironment{abstr}[1]{\vspace*{#1mm}\noindent\bf Abstract\;\rm}{}
\newenvironment{keyw}[1]{\vspace*{#1mm}\noindent\bf Key words\;\rm}{}
\newenvironment{amscl}[1]{\vspace*{#1mm}\noindent\bf AMS MSC-Classifications\;\rm}{}
\newenvironment{acknow}[1]{\vspace*{#1mm}\noindent\bf Acknowledgements\;\rm}{}

\def\reals{\mathbb{R}}

\def\p{\partial}

\newcommand{\sqrts}[1]{#1^{\!\nicefrac{1}{2}}}

\def\Amat{\alpha}
\def\L{\Lambda}

\def\ol{\overline}
\def\ubr{\underbrace}

\def\ds{\displaystyle}

\newcommand{\set}[2]{\{#1\,\mid\,#2\}}

\DeclareMathOperator{\Max}{M}
\DeclareMathOperator{\A}{A}

\def\As{\A^{*}}

\DeclareMathOperator{\ed}{d}
\DeclareMathOperator{\cd}{\delta}
\def\na{\nabla}
\def\nas{\na_{\!\!\mathrm{s}}\,}
\DeclareMathOperator{\rot}{rot}

\DeclareMathOperator{\opdiv}{div}
\def\div{\opdiv}
\DeclareMathOperator{\Div}{Div}
\def\Divs{\Div_{\mathrm{s}}\,}
\DeclareMathOperator{\id}{id}
\DeclareMathOperator{\sym}{sym}

\DeclareMathOperator{\R}{Q}

\def\om{\Omega}
\def\ga{\Gamma}
\def\gad{\ga_{\mathtt{D}}}
\def\gan{\ga_{\mathtt{N}}}
\def\gar{\ga_{\mathtt{R}}}
\def\eps{\epsilon}
\def\alphao{\alpha_{1}}
\def\alphat{\alpha_{2}}

\def\Nelem{N_{\mathtt{elem}}}

\def\ut{\tilde{u}}
\def\pt{\tilde{p}}
\def\xt{\tilde{x}}
\def\yt{\tilde{y}}
\def\Et{\tilde{E}}
\def\Ht{\tilde{H}}
\def\uh{u_{\mathtt{h}}}
\def\ph{p_{\mathtt{h}}}
\def\Eh{E_{\mathtt{h}}}
\def\Hh{H_{\mathtt{h}}}
\def\Eho{E_{\mathtt{h},1}}
\def\Eht{E_{\mathtt{h},2}}
\def\uiter{u_{\mathtt{iter}}}
\def\piter{p_{\mathtt{iter}}}
\def\pav{p_{\mathtt{avg}}}

\def\M{\mathcal{M}}

\def\Mna{\M_{\na}}
\def\Mdiv{\M_{\div}}
\def\Mmix{\M_{\mathrm{mix}}}
\def\Mrd{\M_{\mathrm{rd}}}
\def\Mec{\M_{\mathrm{ec}}}

\def\Mle{\M_{\mathrm{le}}}
\def\Mdiff{\M_{\mathrm{diff}}}

\newcommand{\ttr}[1]{\tau_{#1}}  
\def\ttrga{\ttr{\ga}}
\def\ttrgad{\ttr{\gad}}
\def\ttrgan{\ttr{\gan}}
\newcommand{\ntr}[1]{\nu_{#1}}  
\def\ntrga{\ntr{\ga}}

\def\ntrgan{\ntr{\gan}}

\DeclareMathOperator{\hilbert}{\sf H}
\DeclareMathOperator{\cont}{\sf C}
\DeclareMathOperator{\lebesgue}{\sf L}
\DeclareMathOperator{\rotation}{\sf R}
\DeclareMathOperator{\divergence}{\sf D}
\DeclareMathOperator{\diffform}{\sf D}
\DeclareMathOperator{\codiffform}{\Delta}

\def\hio{\hilbert_{1}}
\def\hit{\hilbert_{2}}

\newcommand{\cgen}[2]{\cont^{#1}_{#2}}
\def\cic{\cgen{\infty}{\circ}}
\def\ciga{\cgen{\infty}{\ga}}
\def\cigad{\cgen{\infty}{\gad}}
\def\cigan{\cgen{\infty}{\gan}}

\newcommand{\lgen}[2]{\lebesgue^{#1}_{#2}}
\def\lt{\lgen{2}{}}
\def\li{\lgen{\infty}{}}

\newcommand{\hgen}[2]{\hilbert^{#1}_{#2}}
\def\ho{\hgen{1}{}}
\def\hoc{\hgen{1}{\circ}}
\def\hoga{\hgen{1}{\ga}}

\def\hogad{\hgen{1}{\gad}}
\def\hogan{\hgen{1}{\gan}}
\def\htwo{\hgen{2}{}}
\def\hoh{\hgen{\nicefrac{1}{2}}{}}

\def\hmoh{\hgen{-\nicefrac{1}{2}}{}}

\newcommand{\dgen}[2]{\divergence^{#1}_{#2}}
\def\d{\dgen{}{}}
\def\dz{\dgen{}{0}}

\def\dgad{\dgen{}{\gad}}
\def\dzgad{\dgen{}{\gad,0}}
\def\dgan{\dgen{}{\gan}}
\def\dzgan{\dgen{}{\gan,0}}

\newcommand{\rgen}[2]{\rotation^{#1}_{#2}}
\def\r{\rgen{}{}}
\def\rz{\rgen{}{0}}

\def\rgad{\rgen{}{\gad}}
\def\rzgad{\rgen{}{\gad,0}}
\def\rgan{\rgen{}{\gan}}
\def\rzgan{\rgen{}{\gan,0}}

\def\ciqgad{\cgen{\infty,q}{\gad}}
\def\ciqgan{\cgen{\infty,q}{\gan}}

\def\ltq{\lgen{2,q}{}}
\def\ltqpo{\lgen{2,q+1}{}}
\def\ltqmo{\lgen{2,q-1}{}}

\newcommand{\diffgen}[2]{\diffform^{#1}_{#2}}
\def\dq{\diffgen{q}{}}
\def\dqz{\diffgen{q}{0}}

\def\dqgad{\diffgen{q}{\gad}}
\def\dqzgad{\diffgen{q}{\gad,0}}

\def\dqpozgad{\diffgen{q+1}{\gad,0}}

\newcommand{\degen}[2]{\codiffform^{#1}_{#2}}
\def\deq{\degen{q}{}}

\def\deqgan{\degen{q}{\gan}}
\def\deqzgan{\degen{q}{\gan,0}}
\def\deqpo{\degen{q+1}{}}

\def\deqpogan{\degen{q+1}{\gan}}
\def\deqzgan{\degen{q}{\gan,0}}

\def\harmdigadgan{\mathcal{H}_{\gad,\gan}}
\def\harmdigangad{\mathcal{H}_{\gan,\gad}}
\def\harmdiqgadgan{\mathcal{H}^q_{\gad,\gan}}
\def\harmdiqpogadgan{\mathcal{H}^{q+1}_{\gad,\gan}}

\newcommand{\norm}[1]{|#1|}

\newcommand{\normhio}[1]{\norm{#1}_{\hio}}

\newcommand{\normhioao}[1]{\norm{#1}_{\hio,\alphao}}

\newcommand{\normhioaomo}[1]{\norm{#1}_{\hio,\alphao^{-1}}}

\newcommand{\normhitat}[1]{\norm{#1}_{\hit,\alphat}}

\newcommand{\normhitatmo}[1]{\norm{#1}_{\hit,\alphat^{-1}}}

\newcommand{\normdaidaoat}[1]{\norm{#1}_{D(\A),\alphao,\alphat}}

\newcommand{\normdasaomoatmo}[1]{\norm{#1}_{D(\As),\alphao^{-1},\alphat^{-1}}}

\newcommand{\tnorm}[1]{\|#1\|}

\newcommand{\normlt}[1]{\norm{#1}_{\lt}}	
\newcommand{\normho}[1]{\norm{#1}_{\ho}}	
\newcommand{\normd}[1]{\norm{#1}_{\d}}	
	
\newcommand{\normltq}[1]{\norm{#1}_{\ltq}}

\newcommand{\scp}[2]{\langle#1,#2\rangle}

\newcommand{\scphio}[2]{\scp{#1}{#2}_{\hio}}

\newcommand{\scphioao}[2]{\scp{#1}{#2}_{\hio,\alphao}}
\newcommand{\scphioaomo}[2]{\scp{#1}{#2}_{\hio,\alphao^{-1}}}
\newcommand{\scphit}[2]{\scp{#1}{#2}_{\hit}}

\newcommand{\scphitat}[2]{\scp{#1}{#2}_{\hit,\alphat}}
\newcommand{\scphitatmo}[2]{\scp{#1}{#2}_{\hit,\alphat^{-1}}}

\newcommand{\scplt}[2]{\scp{#1}{#2}_{\lt}}
\newcommand{\scpltq}[2]{\scp{#1}{#2}_{\ltq}}
\newcommand{\scpltqpo}[2]{\scp{#1}{#2}_{\ltqpo}}

\title{\sc Functional A Posteriori Error Equalities\\ for Conforming Mixed Approximations\\ of Elliptic Problems}

\author{\sf Immanuel Anjam\footnote{\tt Department of Mathematical Information Technology, 
University of Jyv\"askyl\"a, 
P.O. Box 35 (Agora), 
FI-40014 University of Jyv\"askyl\"a, 
Finland,
email: immanuel.anjam@jyu.fi}
\quad\&\quad
Dirk Pauly\footnote{\tt Fakult\"at f\"ur Mathematik,
Universit\"at Duisburg-Essen, 
Campus Essen,
Thea-Leymann-Str. 9, 
DE-45141 Essen, 
Germany, 
email: dirk.pauly@uni-due.de}
}

\begin{document}

\maketitle

\begin{dedi}{0}
Dedicated to Sergey Igorevich Repin on the occasion of his $60^{\textrm{th}}$ birthday
\end{dedi}

\begin{abstr}{12}
In this paper we show how to find the {\it exact error}
(not just an estimate of the error)
of a conforming mixed approximation by using the functional type
a posteriori error estimates in the spirit of Repin \cite{repinbookone}. 
The error is measured in a mixed norm which takes into account both
the primal and dual variables. We derive this result 
for all elliptic partial differential equations of the class
$$\As\A\,x+x=f,$$
where $\A$ is a linear, densely defined and closed (usually a differential) operator 
and $\As$ its adjoint. 
We first derive a special version of our main result by using a simplified reaction-diffusion problem 
to demonstrate the strong connection to the classical functional a posteriori error estimates 
of Repin \cite{repinbookone}. After this we derive the main result in an abstract setting. 
Our main result states that in order to obtain the {\it exact global error} value 
of a conforming mixed approximation with primal variable $x$ 
and dual variable $y$, i.e.,
$$\As y+x=f,\quad\A x=y,$$ 
one only needs the problem data and the mixed approximation $(\xt,\yt)\in D(\A)\times D(\As)$ 
of the exact solution $(x,y)\in D(\A)\times\big(D(\As)\cap R(\A)\big)$, i.e.,
the {\it equality}
$$\norm{x-\xt}^2
+\norm{\A(x-\xt)}^2
+\norm{y-\yt}^2
+\norm{\As(y-\yt)}^2
=\norm{f-\xt-\As\yt}^2
+\norm{\yt-\A\xt}^2$$
holds. There is no need for calculating any auxiliary data. 
The calculation of the exact error consists of simply calculating two (usually integral) quantities 
where all the quantities are known after the approximate solution has been obtained
by any conforming method guaranteeing $(\xt,\yt)\in D(\A)\times D(\As)$. 
We also show some numerical computations to confirm the results.
\end{abstr}

\begin{keyw}{0}
functional a posteriori error estimate, error equality, 
elliptic boundary value problem, mixed formulation, combined norm
\end{keyw}

\begin{amscl}{0}
65N15
\end{amscl}

\newpage
\tableofcontents


\section{Introduction}

The results presented in this paper are based on the conception of functional type a posteriori error estimates. 
These type estimates are valid for any conforming approximation and contain only global constants. 
We note that estimates for nonconforming approximations are known as well but will not be discussed in this paper.
In the case of the class of PDEs studied in this paper, the estimates do not contain even global constants. 
For a detailed exposition of the theory see the books \cite{repinbookone} 
by Repin and \cite{NeittaanmakiRepin2004} by Repin and Neittaanm\"aki
or for a more computational point of view \cite{MaliRepinNeittaanmaki2014}
by Mali, Repin, and Neittaanm\"aki.

We will measure the error of our approximations in a combined norm, which includes the error of both, 
the primal and the dual variable. This is especially useful for mixed methods 
where one calculates an approximation for both the primal and dual variables, see e.g.
the book of Brezzi and Fortin \cite{brezzifortinbookone}.

In this paper, we study the linear equation 
$$(\As\alphat\A+\alphao)x=f$$
presented in the mixed form
$$\As y+\alphao x=f,\quad\alphat\A=y,$$
where $\alphao,\alphat$ are linear and self adjoint topological isomorphisms
on two Hilbert spaces $\hio$ and $\hit$
and $\A:D(\A)\subset\hio\to\hit$ is a linear, densely defined and closed operator
with adjoint operator $\As:D(\As)\subset\hit\to\hio$.
Our main result is Theorem \ref{thm:Gmain} and it shortly reads 
as the {\it functional a posterior error equality}
\begin{align}
\label{erroreq}
\begin{split}
&\qquad\normhioao{x-\xt}^2
+\normhitat{\A(x-\xt)}^2
+\normhitatmo{y-\yt}^2
+\normhioaomo{\As(y-\yt)}^2\\
&=\normhioaomo{f-\alphao\xt-\As\yt}^2
+\normhitatmo{\yt-\alphat\A\xt}^2
\end{split}
\end{align}
being valid for any conforming mixed approximation $(\xt,\yt)\in D(\A)\times D(\As)$ 
of the exact solution $(x,y)\in D(\A)\times D(\As)$.

Functional a posteriori error estimates for combined norms 
were first exposed in the paper \cite{repinsautersmolianskiaposttwosideell},
where the authors present two-sided estimates bounding the error
by the same quantity from below and from above aside from multiplicative constants. 
Unlike in other estimates, these constants are $1$ and $\sqrt{3}$.
In \cite{repinsautersmolianskiaposttwosideell} the authors studied problems of the type 
\begin{align}
\label{probone}
\As\alpha\A x=f ,
\end{align}
i.e., the case $\alpha=\alphat$, $\alphao=0$.


The paper is organized as follows. 
In Section \ref{sec:M} we prove our main results for a simple model problem 
and show the strong connection to the classical functional a posteriori error estimates. 
In Section \ref{sec:G} we derive our main results in an abstract Hilbert space setting and
in Section \ref{sec:A} we show applications of the general results to several classical problems. 
Section \ref{sec:inhomobc} is devoted to inhomogeneous boundary conditions and
finally in Section \ref{sec:N} we present some numerical experiments to confirm our theoretical results.

\section{Results for a Model Problem} \label{sec:M}

Let $\om\subset\reals^d$, $d\geq1$, be open and without loss of generality connected,
so let $\om$ be a domain with boundary $\ga:=\p\om$. 
We emphasize that $\om$ may be bounded or unbounded, like an exterior domain, or non of both.
Moreover, $\ga$ does not need to have any smoothness.
We denote by $\scplt{\,\cdot\,}{\,\cdot\,}$ and $\normlt{\,\cdot\,}$ the inner product and the norm 
in $\lt$ for scalar-, vector- and matrix-valued functions. 
Throughout the paper we will not indicate the dependence on $\om$
in our notations of the functional spaces.
Moreover, we define the usual Sobolev spaces
$$\ho:=\set{\varphi\in\lt}{\na\varphi\in\lt},\quad\d:=\set{\psi\in\lt}{\div\psi\in\lt}$$
and as the closure of smooth and compactly supported test functions\footnote{The 
spaces $\ciga$ and $\hoga$ are often denoted by $\cic$ and $\hoc$.}
$$\hoga:=\ol{\ciga}^{\ho}.$$
These are Hilbert spaces equipped with the respective graph norms denoted by
$\normho{\,\cdot\,}$, $\normd{\,\cdot\,}$.

Our simple model reaction-diffusion problem reads as follows: 
Find the potential $u\in\hoga$, i.e., the primal variable, such that
\begin{align}
-\Delta u+u=-\div\na u+u=f,\label{eq:Ipde}
\end{align}
where $f \in \lt$ is the source term. The variational formulation
of this problem consists of finding $u\in\hoga$ such that
\begin{align} 
\label{eq:Igen}
\forall\,\varphi\in\hoga\qquad
\scplt{\na u}{\na\varphi}+\scplt{u}{\varphi}=\scplt{f}{\varphi}.
\end{align}
The natural energy norm for this problem is $\normho{\,\cdot\,}$.
Of course, by the Lax-Milgram lemma or Riesz' representation theorem 
\eqref{eq:Igen} has a unique solution $u\in\hoga$ satisfying 
$$\normho{u}\leq\normlt{f}.$$
Often, a variable of interest is also the flux, i.e., the dual variable,
$$p:=\na u\in\d,$$
leading to the mixed formulation
$$-\div p+u=f,\quad\na u=p.$$ 
We note that indeed by \eqref{eq:Igen} the flux $p$ belongs to $\d$ and $\div p=u-f$ holds.
Let us further emphasize that even
$$p\in\d\cap\,\na\hoga$$
holds, this is, $p$ is also irrotational, has got vanishing tangential trace 
and is $\lt$-perpendicular to the so-called Dirichlet fields.

We will understand a pair $(\ut,\pt)\in\hoga\times\d$ 
without further requirements as an approximation 
of the exact solution pair $(u,p)\in\hoga\times\d$. 
For the convenience of the reader, 
we first present the classical functional error upper bounds,
frequently called \emph{error majorants}, for the approximations of $u$ and $p$. 


\begin{theo} 
\label{thm:Iest}
For any approximation $\ut\in\hoga$ of the exact potential $u$ 
\begin{align}
\normho{u-\ut}^2=\min_{\psi\in\d}\Mna(\ut,\psi)=\Mna(\ut,p),\label{eq:Iest}
\end{align}
holds, where
\begin{align} 
\label{eq:Imaj}
\Mna(\ut,\psi):=\normlt{f-\ut+\div\psi}^2+\normlt{\psi-\na\ut}^2.
\end{align}
\end{theo}

\begin{proof}
To derive the upper bound, we subtract $\scplt{\na\ut}{\na\varphi}+\scplt{\ut}{\varphi}$ 
from both sides of the generalized form \eqref{eq:Igen}, and obtain for all $\varphi\in\hoga$
\begin{align} 
\label{eq:Imaj_1}
\scplt{\na(u-\ut)}{\na \varphi}+\scplt{u-\ut}{\varphi}
=\scplt{f-\ut}{\varphi}-\scplt{\na\ut}{\na\varphi}.
\end{align}
For an arbitrary function $\psi\in\d$ and any $\varphi\in\hoga$  
we have $\scplt{\div\psi}{\varphi}+\scplt{\psi}{\na\varphi}=0$.
By adding this to the right hand side of \eqref{eq:Imaj_1} it becomes
\begin{align} 
\label{eq:Imaj_2}
\scplt{\na(u-\ut)}{\na \varphi}+\scplt{u-\ut}{\varphi}
&=\scplt{f-\ut+\div\psi}{\varphi}+\scplt{\psi-\na\ut}{\na\varphi}\nonumber\\
& \leq\normlt{f-\ut+\div\psi}\normlt{\varphi}
+\normlt{\psi-\na\ut}\normlt{\na\varphi}\\
&	\leq\Mna(\ut,\psi)^{\nicefrac{1}{2}}\normho{\varphi}.\nonumber
\end{align}
By choosing $\varphi:=u-\ut\in\hoga$  we obtain 
$\normho{u-\ut}^2\leq\Mna(\ut,\psi)$.
Since $p\in\d$, we see that $\Mna(\ut,p)=\normho{u-\ut}^2$.
\end{proof}

As the majorant $\Mna$ is sharp,
it immediately provides a technique to obtain approximations for the exact flux $p$. 
Minimizing $M_{\na}(\psi):=\Mna(\ut,\psi)$ with respect to $\psi$
yields by differentiation for all $\psi\in\d$
\begin{align*}
0\overset{!}{=}M_{\na}'(p)\psi&=2\scplt{f-\ut+\div p}{\div\psi}+2\scplt{p-\na\ut}{\psi}\\
&=2\scplt{f+\div p}{\div\psi}+2\scplt{p}{\psi}
\end{align*}
since $\scplt{\ut}{\div\psi}=-\scplt{\na\ut}{\psi}$ because $\ut\in\hoga$.
Hence the following problem occurs: Find $p\in\d$ such that
\begin{align} 
\label{eq:Igenp}
\forall\psi\in\d\qquad
\scplt{\div p}{\div\psi}+\scplt{p}{\psi}=-\scplt{f}{\div\psi}.
\end{align}
Note that $\ut$ is not present here and the natural energy norm for this problem is $\normd{\,\cdot\,}$.
Once again, by the Lax-Milgram lemma \eqref{eq:Igenp} has a unique solution $p\in\d$ satisfying 
$$\normd{p}\leq\normlt{f}.$$
Since $\na u\in\d$ solves \eqref{eq:Igenp}, i.e., with \eqref{eq:Ipde}
$$\scplt{\div\na u}{\div\psi}
=\scplt{u}{\div\psi}
-\scplt{f}{\div\psi}
=-\scplt{\na u}{\psi}
-\scplt{f}{\div\psi},$$
we get indeed $p=\na u$.

\begin{rem}
\mbox{}
\begin{itemize}
\item[\bf(i)]
The variational formulation \eqref{eq:Igenp} for $p$ 
can also be achieved by testing \eqref{eq:Ipde} with $\div\psi$ for all $\psi\in\d$ since
\begin{align*}
-\scplt{f}{\div\psi}
&=\scplt{\div\na u}{\div\psi}-\scplt{u}{\div\psi}\\
&=\scplt{\div\na u}{\div\psi}+\scplt{\na u}{\psi}
=\scplt{\div p}{\div\psi}+\scplt{p}{\psi}.
\end{align*}
\item[\bf(ii)]
By \eqref{eq:Igenp}
$$p\,\bot\,\dz:=\set{v\in\d}{\div v=0}$$
holds. Thus, by the Helmholtz decomposition, i.e., $\lt=\ol{\na\hoga}\oplus\dz$,
we get $p\in\ol{\na\hoga}$. 
Here, $\bot$ and $\oplus$ denote orthogonality and the orthogonal sum in $\lt$.
\item[\bf(iii)]
\eqref{eq:Igenp} is the dual problem to \eqref{eq:Igen} and its strong formulation 
in duality to \eqref{eq:Ipde} is
\begin{align}
\label{nadivpeq}
-\na\div p+p=\na f
\end{align}
with mixed formulation
$$\na v+p=\na f,\quad-\div p=v.$$
We note that in general $\div p$ does not belong to $\hoga$, not even to $\ho$.
On the other hand, by \eqref{eq:Igenp} we see $\div p+f\in\hoga$ with $\na(\div p+f)=p$
and the natural Neumann boundary condition $\div p+f=0$ at $\ga$ appears.
Hence $f\in\hoga$, if and only if $v:=-\div p\in\hoga$,
and $f\in\ho$, if and only if $v\in\ho$. In both cases \eqref{nadivpeq} holds
and moreover for all $\varphi\in\hoga$
$$\scplt{\na v}{\na\varphi}+\scplt{v}{\varphi}
=-\scplt{p}{\na\varphi}+\scplt{v}{\varphi}+\scplt{\na f}{\na\varphi}
=\scplt{\na f}{\na\varphi},$$
thus $v\in\ho$ solves in the strong sense $-\Delta v+v=-\Delta f$ and $v=f$ at $\ga$
if $\Delta f\in\lt$.
\end{itemize}
\end{rem}


\begin{theo} 
\label{thm:Iestp}
For any approximation $\pt\in\d$ of the exact flux $p$
\begin{align}
\label{eq:Iestp}
\normd{p-\pt}^2=\min_{\varphi\in\hoga}\Mdiv(\pt,\varphi)=\Mdiv(\pt,u),
\end{align}
holds, where
\begin{align} 
\label{eq:Imajp}
\Mdiv(\pt,\varphi):=\normlt{f-\varphi+\div\pt}^2+\normlt{\pt-\na\varphi}^2.
\end{align}
\end{theo}

\begin{proof}
We add $-\scplt{\div\pt}{\div\psi}-\scplt{\pt}{\psi}$ to the both sides 
of the variational formulation \eqref{eq:Igenp} and obtain for all $\psi\in\d$
\begin{align}\label{eq:Imajp_1}
\scplt{\div(p-\pt)}{\div\psi}+\scplt{p-\pt}{\psi}=
-\scplt{f+\div\pt}{\div\psi}-\scplt{\pt}{\psi}.
\end{align}
For any $\varphi\in\hoga$ we have again $\scplt{\na\varphi}{\psi}+\scplt{\varphi}{\div\psi}=0$. 
By adding this to the right hand side of \eqref{eq:Imajp_1} it becomes
\begin{align} 
\label{eq:Imajp_2}
\scplt{\div(p-\pt)}{\div \psi}+\scplt{p-\pt}{\psi}
&=-\scplt{f-\varphi+\div\pt}{\div\psi}-\scplt{\pt-\na\varphi}{\psi}\nonumber\\
&\leq\normlt{f-\varphi+\div\pt}\normlt{\div\psi}
+\normlt{\pt-\na\varphi}\normlt{\psi}\\
&\leq\Mdiv(\pt,\varphi)^{\nicefrac{1}{2}}\normd{\psi}.\nonumber
\end{align}
Choosing $\psi=p-\pt\in\d$ yields $\normd{p-\pt}^2\leq\Mdiv(\pt,\varphi)$.
Finally $\Mdiv(\pt,u)=\normd{p-\pt}^2$ follows by $u\in\hoga$.
\end{proof}

As before, the sharpness of the majorant $\Mdiv$ gives us a technique
to obtain approximations of the potential $u$. 
In fact, now global minimization of $M_{\div}(\varphi):=\Mdiv(\pt,\varphi)$ 
with respect to $\varphi$ would lead to the variational formulation \eqref{eq:Igen}
for finding $u$, since for all $\varphi\in\hoga$
\begin{align*}
0\overset{!}{=}M_{\div}'(u)\varphi&=-2\scplt{f-u+\div\pt}{\varphi}-2\scplt{\pt-\na u}{\na\varphi}\\
&=2\scplt{u-f}{\varphi}+2\scplt{\na u}{\na\varphi}
\end{align*}
since $\scplt{\div\pt}{\varphi}=-\scplt{\pt}{\na\varphi}$ by $\pt\in\d$.
 
Finally, we note that the functional a posteriori error majorants $\Mna$ and $\Mdiv$ 
contain only the problem data, conforming numerical approximations
and the free functions $\psi$ and $\varphi$.


Now, we define the combined norm for the reaction-diffusion problem 
in a canonical way as the sum of the energy norms for the potential and the flux:
\begin{equation*}
\tnorm{(\varphi,\psi)}^2 :=
\normho{\varphi}^2+\normd{\psi}^2
=\normlt{\varphi}^2+\normlt{\na\varphi}^2 
+\normlt{\psi}^2+\normlt{\div\psi}^2
\end{equation*}

\begin{rem}
We know $\normho{u}\leq\normlt{f}$ and $\normd{p}\leq\normlt{f}$.
It is indeed notable that 
$$\tnorm{(u,p)}=\normlt{f}$$ 
holds, which follows immediately by $f=-\div p+u$ and $p=\na u$ since
$$\normlt{f}^2
=\normlt{\div p}^2
+\normlt{u}^2
-2\scplt{\div p}{u}\\
=\normlt{\div p}^2
+\normlt{u}^2
+2\normlt{p}^2
=\tnorm{(u,p)}^2.$$
Hence the solution operator
$$L:\lt\to\hoga\times\d;f\mapsto(u,p)$$
has norm $\norm{L}=1$, i.e., $L$ is an isometry.
\end{rem}

Our main result for this simple reaction-diffusion problem 
basically combines Theorems \ref{thm:Iest} and \ref{thm:Iestp}.
However, we outline that the resulting right hand side does not contain $u$ or $p$ anymore
and is even an equality.

\begin{theo} 
\label{thm:Imain}
For any approximation $(\ut,\pt)\in\hoga\times\d$ 
of the exact solution $(u,p)$
\begin{align}
\tnorm{(u,p)-(\ut,\pt)}^2
&=\Mmix(\ut,\pt)
\label{eq:Imain1}
\intertext{and the normalized counterpart}
\frac{\tnorm{(u,p)-(\ut,\pt)}^2}{\tnorm{(u,p)}^2}
&=\frac{\Mmix(\ut,\pt)}{\normlt{f}^2}
\label{eq:Imain2}
\end{align}
hold, where
\begin{align} 
\label{eq:Imain}
\begin{split}
\Mmix(\ut,\pt)
&:=\Mna(\ut,\pt)=\Mdiv(\pt,\ut)
=\normlt{f-\ut+\div\pt}^2+\normlt{\pt-\nabla \ut}^2.
\end{split}
\end{align}
\end{theo}

The error in the combined norm can thus be exactly computed by quantities we already know: 
the given problem data $f$ and the conforming approximation $(\ut,\pt)$.

\begin{proof}
Set $\psi=\pt$ in \eqref{eq:Imaj_2} and $\varphi=\ut$ in \eqref{eq:Imajp_2}.
Then, for any $\varphi\in\hoga$ and any $\psi\in\d$ we have
\begin{align} 
\label{eq:Imain_1}
\scplt{\na(u-\ut)}{\na \varphi}+\scplt{u-\ut}{\varphi}
&=\scplt{f-\ut+\div\pt}{\varphi}+\scplt{\pt-\na\ut}{\na\varphi},\\
\label{eq:Imain_2}
\scplt{\div (p-\pt)}{\div\psi}+\scplt{p-\pt}{\psi}
&=-\scplt{f-\ut+\div\pt}{\div\psi}-\scplt{\pt-\na\ut}{\psi}.
\end{align}
Adding \eqref{eq:Imain_1} and \eqref{eq:Imain_2} we obtain
\begin{align}
\label{eq:main_3}
\begin{split}
&\qquad\scplt{\na(u-\ut)}{\na\varphi}+\scplt{u-\ut}{\varphi}
+\scplt{\div(p-\pt)}{\div\psi}+\scplt{p-\pt}{\psi}\\
&=\scplt{f-\ut+\div\pt}{\varphi-\div\psi}+\scplt{\pt-\na\ut}{\na\varphi-\psi}.
\end{split}
\end{align}
By choosing $\varphi:=u-\ut\in\hoga$ and $\psi:=p-\pt\in\d$, 
the left hand side of \eqref{eq:main_3} turns to the combined norm 
of the error of the approximation. Since we have
\begin{align*}
\varphi-\div\psi
&=u-\ut-\div p+\div\pt
=f-\ut+\div\pt,\\
\na\varphi-\psi
&=\na u-\na\ut-p+\pt
=\pt-\na\ut,
\end{align*}
\eqref{eq:main_3} becomes \eqref{eq:Imain1}.
Putting $\ut=0$, $\pt=0$ in \eqref{eq:Imain1} shows $\tnorm{(u,p)}=\normlt{f}$
and thus \eqref{eq:Imain2}.
\end{proof}

\begin{rem}
\label{thm:Imainrem}
\mbox{}
\begin{itemize}
\item[\bf(i)]
We note the similarity of the error majorants in 
Theorems \ref{thm:Iest}, \ref{thm:Iestp} and \ref{thm:Imain}.
\item[\bf(ii)] 
It is clear that Theorem \ref{thm:Imain} generalizes Theorems \ref{thm:Iest} and \ref{thm:Iestp}
since these two can be recovered from Theorem \ref{thm:Imain}.
We just estimate
\begin{align*}
\Mna(\ut,p)
=\normho{u-\ut}^2
&\leq\tnorm{(u,p)-(\ut,\pt)}^2
=\Mmix(\ut,\pt)=\M_\na(\ut,\pt)
\intertext{and note that the left hand side does not depend on $\psi:=\pt\in\d$.
Analogously we estimate}
\Mdiv(\pt,u)
=\normd{p-\pt}^2
&\leq\tnorm{(u,p)-(\ut,\pt)}^2
=\Mmix(\ut,\pt)=\Mdiv(\pt,\ut)
\end{align*}
and note that the left hand side does not depend on $\varphi:=\ut\in\hoga$.
\end{itemize}
\end{rem}

\begin{rem} 
\label{rem:Iproof}
There is a simple proof of Theorem \ref{thm:Imain}
using just \eqref{eq:Ipde} and $p=\na u$:
\begin{align*}
\Mmix(\ut,\pt)
&=\normlt{f-\ut+\div\pt}^2+\normlt{\pt-\na\ut}^2\\
&=\normlt{u-\ut+\div\pt-\div p}^2+\normlt{\pt-p+\na u-\na\ut}^2\\
&=\normlt{u-\ut}^2+\normlt{\div(\pt-p)}^2+2\scplt{u-\ut}{\div(\pt-p)}\\
&\qquad+\normlt{\pt-p}^2+\normlt{\na(u-\ut)}^2+2\scplt{\pt-p}{\na(u-\ut)}\\
&=\tnorm{(u,p)-(\ut,\pt)}^2
\end{align*}
In the last line we have used as before 
$\scplt{u - \ut}{\div(\pt - p)}=-\scplt{\na(u-\ut)}{\pt - p}$
since $u-\ut\in\hoga$.
This shows immediately, that Theorem \ref{thm:Imain} 
extends to more general situations as well.
E.g. inhomogeneous boundary conditions can be treated since
only $u-\ut\in\hoga$ is needed. 
\end{rem}


\section{Results for the General Case} \label{sec:G}

In this section we derive our main result in an abstract setting which
allows for mixed boundary conditions as well as coefficients for the PDEs. 
We will prove the main result by using the simple approach presented in Remark \ref{rem:Iproof}.

Let ${\hio}$ and ${\hit}$ be two Hilbert spaces with inner products 
$\scphio{\,\cdot\,}{\,\cdot\,}$ and $\scphit{\,\cdot\,}{\,\cdot\,}$, respectively. 
Moreover, let $\A:D(\A)\subset\hio\to\hit$ be a densely defined and closed linear operator
and $\As:D(\As)\subset\hit\to\hio$ its adjoint. We note $\A^{**}=\bar{\A}=\A$ and 
\begin{align}
\label{partint}
\forall\,\varphi\in D(\A)\quad\forall\,\psi\in D(\As)\qquad
\scphit{\A\varphi}{\psi}=\scphio{\varphi}{\As\psi}.
\end{align}
Equipped with the natural graph norms $D(\A)$ and $D(\As)$ are Hilbert spaces. 
Furthermore, we introduce two linear, self adjoint and positive topological isomorphisms 
$\alphao:{\hio}\to{\hio}$ and $\alphat:{\hit}\to{\hit}$. Especially we have
$$\exists\,c>0\quad\forall\,\varphi\in\hio\qquad 
c^{-1}\normhio{\varphi}^2\leq\scphio{\alphao\varphi}{\varphi}\leq c\normhio{\varphi}^2$$
and the corresponding holds for $\alphat$.
For any inner product and corresponding norm we introduce weighted counterparts with sub-index notation. 
For example, for elements from ${\hio}$ we define a new inner product 
$\scphioao{\,\cdot\,}{\,\cdot\,}:=\scphio{\alphao\,\cdot\,}{\,\cdot\,}$ 
and a new induced norm $\normhioao{\,\cdot\,}$.
Using this notation we can define for $\varphi\in D(\A)$ and $\psi\in D(\As)$ new weighted norms
on $D(\A)$, $D(\As)$ as well as on the product space $D(\A)\times D(\As)$ by
\begin{align*}
\normdaidaoat{\varphi}^2
&:=\normhioao{\varphi}^2
+\normhitat{\A\varphi}^2,\\
\normdasaomoatmo{\psi}^2
&:=\normhitatmo{\psi}^2
+\normhioaomo{\As\psi}^2,\\
\tnorm{(\varphi,\psi)}^2
&:=\normdaidaoat{\varphi}^2
+\normdasaomoatmo{\psi}^2.
\end{align*}

Let $f \in {\hio}$. By the Lax-Milgram lemma
(or by Riesz' representation theorem) we get immediately:

\begin{lem}
\label{laxmilgramA}
The (primal) variational problem
\begin{align}
\label{varA}
\forall\varphi\in D(\A)\qquad
\scphitat{\A x}{\A\varphi}+\scphioao{x}{\varphi}=\scphio{f}{\varphi}
\end{align}
admits a unique solution $x\in D(\A)$ satisfying $\normdaidaoat{x}\leq\normhioaomo{f}$. 
Moreover, $y_{x}:=\alphat\A x$ belongs to $D(\As)$ and $\As y_{x}=f-\alphao x$.
Hence, the strong and mixed formulations
\begin{align} 
\label{eq:Gstrong}
\As\alphat\A x+\alphao x&=f,\\
\label{mixedformulation}
\As y_{x}+\alphao x&=f,\quad\alphat\A x=y_{x}
\end{align}
hold with $(x,y_{x})\in D(\A)\times \big(D(\As)\times\alphat R(\A)\big)$.
\end{lem}

To get the dual problem, we multiply the first equation of \eqref{mixedformulation}
by $\As\psi$ with $\psi\in D(\As)$ 
taking the right weighted scalar product and use $y_{x}=\alphat\A x\in D(\As)$.
We obtain
$$\scphioaomo{\As y_{x}}{\As\psi}
+\scphioaomo{\alphao x}{\As\psi}
=\scphioaomo{f}{\As\psi}.$$
Since $x\in D(\A)$
$$\scphioaomo{\alphao x}{\As\psi}
=\scphio{x}{\As\psi}
=\scphit{\A x}{\psi}
=\scphitatmo{y_{x}}{\psi}$$
holds, we get again by the Lax-Milgram's lemma

\begin{lem}
\label{laxmilgramAs}
The (dual) variational problem
\begin{align}
\label{varAs}
\forall\psi\in D(\As)\qquad
\scphioaomo{\As y}{\As\psi}
+\scphitatmo{y}{\psi}
=\scphioaomo{f}{\As\psi}
\end{align}
admits a unique solution $y\in D(\As)$ satisfying $\normdasaomoatmo{y}\leq\normhioaomo{f}$.
Moreover, $y=y_{x}$ holds and thus $y$ even belongs to 
$D(\As)\cap\alphat R(\A)$ with $x$ and $y_{x}$ from Lemma \ref{laxmilgramA}.
Furthermore, $\alphao^{-1}(\As y-f)\in D(\A)$ with
$A\alphao^{-1}(\As y-f)=-\alphat^{-1}y$.
\end{lem}

\begin{proof}
We just have to show that $y_{x}\in D(\As)$ solves \eqref{varAs}.
But this follows directly since for all $\psi\in D(\As)$
\begin{align*}
\scphioaomo{\As y_{x}}{\As\psi}
&=-\scphio{x}{\As\psi}
+\scphioaomo{f}{\As\psi}\\
&=-\scphit{\A x}{\psi}
+\scphioaomo{f}{\As\psi}
=-\scphitatmo{y_{x}}{\psi}
+\scphioaomo{f}{\As\psi}.
\end{align*}
Hence $y_{x}=y$ and $\A^{**}=\A$ completes the proof.
\end{proof}

\begin{rem}
\label{rem:isometrygen}
We know $\normdaidaoat{x}\leq\normhioaomo{f}$ and $\normdasaomoatmo{y}\leq\normhioaomo{f}$.
It is indeed notable that 
$$\tnorm{(x,y)}=\normhioaomo{f}$$ 
holds, which follows immediately by $y=\alphat\A x$ and
\begin{align*}
\normhioaomo{f}^2
=\normhioaomo{\As\alphat\A x+\alphao x}^2
&=\normhioaomo{\As y}^2
+\normhioaomo{\alphao x}^2
+2\ubr{\scphioaomo{\As\alphat\A x}{\alphao x}}_{\ds=\scphio{\As\alphat\A x}{x}}\\
&=\normhioaomo{\As y}^2
+\normhioao{x}^2
+2\ubr{\scphit{\alphat\A x}{\A x}}_{\ds=\normhitat{\A x}^2}
=\tnorm{(x,y)}^2.
\end{align*}
Thus the solution operator
$$L:\hio\to D(\A)\times D(\As);f\mapsto(x,y)$$
(equipped with the proper weighted norms)
has norm $\norm{L}=1$, i.e., $L$ is an isometry.
\end{rem}

By the latter remark the mixed norm on $D(\A)\times D(\As)$ yields an isomtery.
This motivates to use the mixed norm also for error estimates.
As it turns out, we even obtain an error equality.
We present our main result of the paper.

\begin{theo} 
\label{thm:Gmain}
Let $(x,y),(\xt,\yt)\in D(\A)\times D(\As)$ be the exact solution of \eqref{mixedformulation}
and any conforming approximation, respectively. Then
\begin{align} 
\label{eq:Gmain1}
\tnorm{(x,y)-(\xt,\yt)}^2=\M(\xt,\yt)
\end{align}
and the normalized counterpart
\begin{align} 
\label{eq:Gmain2}
\frac{\tnorm{(x,y)-(\xt,\yt)}^2}{\tnorm{(x,y)}^2}
=\frac{\M(\xt,\yt)}{\normhioaomo{f}^2} 
\end{align}
hold, where
\begin{align} 
\label{eq:Gmain}
\M(\xt,\yt):=\normhioaomo{f-\alphao\xt-\As\yt}^2
+\normhitatmo{\yt-\alphat\A\xt}^2.
\end{align}
\end{theo}

\begin{proof}
Using \eqref{eq:Gstrong} and inserting $0=\alphat\A x-y$ we get by \eqref{partint}
\begin{align*}
\M(\xt,\yt)
&=\normhioaomo{\alphao x-\alphao\xt+\As y-\As\yt}^2
+\normhitatmo{\yt-y+\alphat\A x-\alphat\A\xt}^2\\
&=\normhioao{x-\xt}^2
+\normhioaomo{\As(y-\yt)}^2
+2\scphioaomo{\alphao(x-\xt)}{\As(y-\yt)}\\
&\qquad+\normhitatmo{\yt-y}^2
+\normhitat{\A(x-\xt)}^2
+2\scphitatmo{\yt-y}{\alphat\A(x-\xt)}\\
&=\normdaidaoat{x-\xt}^2+\normdasaomoatmo{y-\yt}^2\\
&\qquad+2\scphio{x-\xt}{\As(y-\yt)}
-2\scphit{\A(x-\xt)}{y-\yt}\\
&=\tnorm{(x,y)-(\xt,\yt)}^2.
\end{align*}
\eqref{eq:Gmain2} follows by the isometry property in Remark \ref{rem:isometrygen}, completing the proof.
\end{proof}

We note that the isometry property, i.e., $\tnorm{(x,y)}=\normhioaomo{f}$,
can be seen by inserting $(\xt,\yt)=(0,0)$ into \eqref{eq:Gmain1} as well.

\begin{rem}
Theorem \ref{thm:Gmain} can also be deduced as a special case of the equation \cite[(7.2.14)]{NeittaanmakiRepin2004} in the book of Neittaam\"aki and Repin.
\end{rem}

\begin{rem} 
\label{rem:Gbehavior}
Of course, the majorant $\M$ is continuous. Especially we have
\begin{align*}
\M(\xt,\yt)\xrightarrow{\xt\rightarrow x\text{ in }D(\A)}
&\normdasaomoatmo{y-\yt}^2=\M(x,\yt),\\
\M(\xt,\yt)\xrightarrow{\yt\rightarrow y\text{ in }D(\As)}
&\normdaidaoat{x-\xt}^2=\M(\xt,y)
\end{align*}
and $\M(\xt,\yt)\to\M(x,y)=0$ if $(\xt,\yt)\to(x,y)$ in $D(\A)\times D(\As)$.
This suggests that the majorant $\M$ can also be used as an error indicator for adaptive computations, 
even though the equality \eqref{eq:Gmain1} is global.
\end{rem}

\begin{cor}
\label{cor:maintheocor}
Theorem \ref{thm:Gmain} provides the well known a posteriori error estimates 
for the primal and dual problems. 
\begin{itemize}
\item[\bf(i)]
For any $\xt\in D(\A)$ it holds
$\ds\normdaidaoat{x-\xt}^2
=\min_{\psi\in D(\As)}\M(\xt,\psi)
=\M(\xt,y)$.
\item[\bf(ii)]
For any $\yt\in D(\As)$ it holds
$\ds\normdasaomoatmo{y-\yt}^2
=\min_{\varphi\in D(\A)}\M(\varphi,\yt)
=\M(x,\yt)$.
\end{itemize}
\end{cor}

\begin{proof}
We just have to estimate
$$\normdaidaoat{x-\xt}^2
\leq\tnorm{(x,y)-(\xt,\yt)}^2
=\M(\xt,\yt)$$
and note that the left hand side does not depend on $\yt\in D(\As)$.
Setting $\psi:=\yt\in D(\As)$ we get
$$\normdaidaoat{x-\xt}^2
\leq\inf_{\psi\in D(\As)}\M(\xt,\psi).$$
But for $\psi=y\in D(\As)$ we see $\M(\xt,y)=\normdaidaoat{x-\xt}^2$,
which proves (i). Analogously, we estimate
$$\normdasaomoatmo{y-\yt}^2
\leq\tnorm{(x,y)-(\xt,\yt)}^2
=\M(\xt,\yt)$$
and note that the left hand side does not depend on $\xt\in D(\A)$.
Setting $\varphi:=\xt\in D(\A)$ we get
$$\normdasaomoatmo{y-\yt}^2
\leq\inf_{\varphi\in D(\A)}\M(\varphi,\yt).$$
But for $\varphi=x\in D(\A)$ we see $\M(x,\yt)=\normdasaomoatmo{y-\yt}^2$,
which shows (ii).
\end{proof}

\begin{rem}
\label{laxmilgramAsrem}
\mbox{}
\begin{itemize}
\item[\bf(i)]
Since $y\,\bot_{\alphat^{-1}}\,N(\As)$ by \eqref{varAs} we get immediately $y\in\alphat\ol{R(\A)}$
by the Helmholtz decomposition $\hit=N(\As)\oplus_{\alphat^{-1}}\alphat\ol{R(\A)}$.
\item[\bf(ii)] 
If $\alphao^{-1}f\in D(\A)$ we have $z:=\alphao^{-1}\As y\in D(\A)$ 
and the strong and mixed formulations of \eqref{varAs} read
\begin{align*}
\A\alphao^{-1}\As y+\alphat^{-1}y&=\A\alphao^{-1}f,\\
\A z+\alphat^{-1}y&=\A\alphao^{-1}f,\quad \alphao^{-1}\As y=z.
\end{align*}
Then for all $\varphi\in D(\A)$ we have
\begin{align*}
\scphitat{\A z}{\A\varphi}+\scphioao{z}{\varphi}
&=-\scphit{y}{\A\varphi}+\scphioao{z}{\varphi}+\scphitat{\A\alphao^{-1}f}{A\varphi}\\
&=\scphitat{\A\alphao^{-1}f}{A\varphi}
\end{align*}
and hence $z\in\big(D(\A)\cap\alphao^{-1} R(\As)\big)\subset D(\A)$ 
is the unique solution of this variational problem.
Furthermore, $\alphat(\A z-\A\alphao^{-1}f)\in D(\As)$ 
and $\As\alphat(\A z-\A\alphao^{-1}f)=-\alphao z$.
If $\alphat\A\alphao^{-1}f$ belongs to $D(\As)$ this yields 
$\alphat\A z\in D(\As)$ and the strong equation
$$\As\alphat\A z+\alphao z=\As\alphat\A\alphao^{-1}f.$$
\end{itemize}
\end{rem}

Our error equalities may also be used to compute 
the radius of the indeterminacy set of solutions 
in terms of the radius of the indeterminacy set of right hand sides.
Often the right hand $f$ of a problem is not known exactly  
but known to belong to an indeterminacy ball around some known mean data $\hat{f}$.
Let us write $f=\hat{f}+f_{\mathtt{osc}}$.
Since the solution operator $L$ from Remark \ref{rem:isometrygen} is an isometry, 
we have for the solutions $(x,y)=(\hat{x},\hat{y})+(x_{\mathtt{osc}},y_{\mathtt{osc}})$
$$\tnorm{(x_{\mathtt{osc}},y_{\mathtt{osc}})}
=\tnorm{Lf_{\mathtt{osc}}}
=\normhioaomo{f_{\mathtt{osc}}}.$$
Hence, the solutions belong to a ball of the same radius as the data.
In other words, any modeling error is mapped to an error of same size.
If the magnitude of the oscillating part $f_{\mathtt{osc}}$ is known, we also
know the magnitude of variations of the solution set.

\subsection{Application to Time Discretization}

One main application of our error equalities might be that equations of the type
\begin{align}
\label{AsAeqapp}
\As\alphat\A x+\alphao x=f
\end{align}
naturally occur in many types of time discretizations for plenty of linear wave propagation models.
A large class of wave propagation models, like electro-magnetics, acoustics or elasticity, have the structure 
$$(\p_{t}\Lambda^{-1}+\Max)\begin{bmatrix}x\\y\end{bmatrix}
=\begin{bmatrix}g\\h\end{bmatrix},\quad
\Max=\begin{bmatrix}0&-\As\\\A&0\end{bmatrix},\quad
\Lambda=\begin{bmatrix}\lambda_{1}&0\\0&\lambda_{2}\end{bmatrix}$$
or
\begin{align}
\label{dteq}
\p_{t}\lambda_{1}^{-1}x-\As y=g,\quad
\p_{t}\lambda_{2}^{-1}y+\A x=h
\end{align}
with initial condition $(x,y)(0)=(x_{0},y_{0})$.
Often the material is assumed to be time-independent, i.e., $\Lambda$ does not depend on time.
In this case $i\Lambda\Max$ is selfadjoint in the proper Hilbert spaces 
and the solution theory follows immediately by the spectral theorem.
We note that formally the second order wave equation 
$$\big(\p_{t}^2-(\Lambda\Max)^2\big)\begin{bmatrix}x\\y\end{bmatrix}
=(\p_{t}-\Lambda\Max)\Lambda\begin{bmatrix}g\\h\end{bmatrix},\quad
(\Lambda\Max)^2
=\begin{bmatrix}-\lambda_{1}\As\lambda_{2}\A&0\\0&-\lambda_{2}\A\lambda_{1}\As\end{bmatrix}$$
holds. A standard implizit time discretization for \eqref{dteq} is e.g. the backward Euler scheme, i.e.,
$$\delta_{n}^{-1}\lambda_{1}^{-1}(x_{n}-x_{n-1})-\As y_{n}=g_{n},\quad
\delta_{n}^{-1}(y_{n}-y_{n-1})+\lambda_{2}\A x_{n}=\lambda_{2}h_{n},\quad
\delta_{n}:=t_{n}-t_{n-1}.$$
Hence, we obtain e.g. for $x_{n}$
$$\As\lambda_{2}\A x_{n}+\delta_{n}^{-2}\lambda_{1}^{-1}x_{n}
=f_{n}:=\As(\lambda_{2}h_{n}+\delta_{n}^{-1}y_{n-1})+\delta_{n}^{-2}\lambda_{1}^{-1}x_{n-1} + \delta_n^{-1} g_n$$
provided that $\lambda_{2}h_{n}\in D(\As)$.
Therefore \eqref{AsAeqapp} holds for $x_{n}$ with e.g. 
$\alphao=\delta_{n}^{-2}\lambda_{1}^{-1}$ and $\alphat=\lambda_{2}$.
Of course, a similar equation holds for $y_{n}$ as well.
We note that our arguments extend to `all' practically used time discretizations.

Functional a posteriori error estimates for wave equations can be found 
in \cite{repinapostwave,paulyrepinrossihypmax}.


\section{Applications} \label{sec:A}

We will discuss some standard applications.
Let $\om\subset\reals^d$, $d\geq1$.
Since we want to handle mixed boundary conditions, 
let us assume for simplicity, that $\om$ is a bounded or an exterior domain 
with (compact) Lipschitz continuous boundary $\ga$.
Moreover, let $\gad$ be an open subset of $\ga$ and $\gan:=\ga\setminus\ol{\gad}$
its complement. We will denote by $n$ the outward unit normal of the boundary. 
The results presented in this section are direct consequences of Theorem \ref{thm:Gmain}
and, of course, Lemmas \ref{laxmilgramA}, \ref{laxmilgramAs}
and Remarks \ref{rem:Gbehavior}, \ref{rem:isometrygen}
as well as Corollary \ref{cor:maintheocor} hold for all special applications.


\subsection{Reaction-Diffusion} \label{subsec:RD}

Find the scalar potential $u\in\ho$, such that
\begin{align}
-\div\Amat\na u+\rho\,u&=f&
\textrm{in }&\om,\nonumber\\
u&=0&
\textrm{on }&\gad,\label{eq:RD}\\
n\cdot\Amat\na u&=0&
\textrm{on }&\gan\nonumber.
\end{align}
The quadratic diffusion matrix $\Amat\in\li$ is symmetric, real valued and uniformly positive definite. 
The reaction coefficient $\rho\ge\rho_0>0$ belongs to $\li$
and the source $f$ to $\lt$. The dual variable for this problem is the flux $p=\Amat\na u\in\d$.
We need more Sobolev spaces
$$\hogad:=\ol{\cigad}^{\ho},\quad
\dgan:=\ol{\cigan}^{\d},\quad
\dzgan:=\set{\psi\in\dgan}{\div\psi=0},$$
where $\cigad$ resp. $\cigan$ are smooth test functions resp. vector fields 
having supports bounded away from $\gad$ resp. $\gan$.
The following table shows the relation to the notation of Section \ref{sec:G}.
\begin{center}\begin{tabular}{c|c||c|c||c|c||c|c}
$\alphao$ & $\alphat$ & $\A$ & $\As$ & $\hio$ & $\hit$ & $D(\A)$ & $D(\As)$ \\
\hline
$\rho$ & $\Amat$ & $\na$ & $-\div$ & $\lt$ & $\lt$ & $\hogad$ & $\dgan$
\end{tabular}\end{center}
We note that indeed $D(\As)=\dgan$ holds for Lipschitz domains, see e.g. \cite{jochmanncompembmaxmixbc},
which is not trivial at all. The relation \eqref{partint} reads now
$$\forall\,\varphi\in\hogad\quad\forall\,\psi\in\dgan\qquad
\scplt{\na\varphi}{\psi}=-\scplt{\varphi}{\div\psi}.$$
Considering the norms we have
\begin{align*}
\norm{u}_{\ho,\rho,\Amat}^2
&=\norm{u}_{\lt,\rho}^2
+\norm{\na u}_{\lt,\Amat}^2,\\
\norm{p}_{\d,\rho^{-1},\Amat^{-1}}^2
&=\norm{p}_{\lt,\Amat^{-1}}^2
+\norm{\div p}_{\lt,\rho^{-1}}^2,\\
\tnorm{(u,p)}^2
&=\norm{u}_{\ho,\rho,\Amat}^2
+\norm{p}_{\d,\rho^{-1},\Amat^{-1}}^2.
\end{align*}
Now \eqref{eq:RD} reads: Find $u\in\hogad$ with $\Amat\na u\in\dgan$ such that
\begin{align}
\label{formapprd}
-\div\Amat\na u+\rho\,u=f.
\end{align}
Equivalently, in mixed formulation we have: Find $(u,p)\in\hogad\times\dgan$ such that
\begin{align}
\label{mixedformapprd}
-\div p+\rho\,u=f,\quad\Amat\na u=p.
\end{align}
The primal and dual variational problems are: Find $(u,p)\in\hogad\times\dgan$ such that
\begin{align*}
\forall\,\varphi&\in\hogad&
\scp{\na u}{\na\varphi}_{\lt,\Amat}+\scp{u}{\varphi}_{\lt,\rho}
&=\scplt{f}{\varphi},\\
\forall\,\psi&\in\dgan&
\scp{\div p}{\div\psi}_{\lt,\rho^{-1}}+\scp{p}{\psi}_{\lt,\Amat^{-1}}
&=-\scp{f}{\div\psi}_{\lt,\rho^{-1}}.
\end{align*}

\begin{theo} 
\label{thm:RD}
Let $(u,p),(\ut,\pt)\in\hogad\times\dgan$ 
be the exact solution of \eqref{mixedformapprd} 
and any approximation, respectively. Then
$$\tnorm{(u,p)-(\ut,\pt)}^2
=\Mrd(\ut,\pt),\quad
\frac{\tnorm{(u,p)-(\ut,\pt)}^2}{\tnorm{(u,p)}^2}
=\frac{\Mrd(\ut,\pt)}{\norm{f}_{\lt,\rho^{-1}}^2}$$
hold, where $\Mrd(\ut,\pt)=\norm{f-\rho\ut+\div\pt}_{\lt,\rho^{-1}}^2
+\norm{\pt-\Amat\na\ut}_{\lt,\Amat^{-1}}^2$.
\end{theo}

\begin{rem}
We note $\norm{u}_{\ho,\rho,\Amat}\leq\norm{f}_{\lt,\rho^{-1}}$ 
and $\norm{p}_{\d,\rho^{-1},\Amat^{-1}}\leq\norm{f}_{\lt,\rho^{-1}}$ and indeed
$$\tnorm{(u,p)}=\norm{f}_{\lt,\rho^{-1}}.$$
The solution operator $L:\lt\to\hogad\times\dgan;f\mapsto(u,p)$ is an isometry, i.e. $\norm{L}=1$.
\end{rem}

\begin{cor}
Theorem \ref{thm:RD} provides the well known a posteriori error estimates 
for the primal and dual problems. 
\begin{itemize}
\item[\bf(i)]
For any $\ut\in\hogad$ it holds
$\ds\norm{u-\ut}_{\ho,\rho,\Amat}^2
=\min_{\psi\in\dgan}\Mrd(\ut,\psi)
=\Mrd(\ut,p)$.
\item[\bf(ii)]
For any $\pt\in\dgan$ it holds
$\ds\norm{p-\pt}_{\d,\rho^{-1},\Amat^{-1}}^2
=\min_{\varphi\in\hogad}\Mrd(\varphi,\pt)
=\Mrd(u,\pt)$.
\end{itemize}
\end{cor}

\begin{rem}
We have $p=\Amat\na u\in\dgan\cap\,\Amat\na\hogad$ and $u$ and $(u,p)$ 
solve \eqref{formapprd} and \eqref{mixedformapprd}, respectively.
Moreover, $\div p+f\in\rho\hogad$
with $\na\rho^{-1}(\div p+f)=\Amat^{-1}p\in\na\hogad=\rzgad\cap\,\harmdigadgan^{\bot}$.
Hence, for $f\in\rho\ho$ we have $\div p\in\rho\ho$ and therefore
the strong and mixed formulations of the dual problem
\begin{align*}
-\na\rho^{-1}\div p+\Amat^{-1}p&=\na\rho^{-1}f&
&&
\textrm{in }&\om,\\
\na v+\Amat^{-1}p&=\na\rho^{-1}f,&
-\rho^{-1}\div p&=v&
\textrm{in }&\om
\intertext{hold, which are completed by the equations}
\div p+f&=0&
&&
\textrm{on }&\gad,\\
n\cdot p&=0&
&&
\textrm{on }&\gan,\\
\rot\Amat^{-1}p&=0&
&&
\textrm{in }&\om,\\
n\times\Amat^{-1}p&=0&
&&
\textrm{on }&\gad,\\
\Amat^{-1}p&\,\,\bot\,\,\harmdigadgan.
\end{align*}
Here the Dirichlet-Neumann fields $\harmdigadgan$ and
the space $\rzgad$ will be defined in Section \ref{subsec:EC}.
Of course, $\rho v=f$ on $\gad$
and by $\rho v\in\div\dgan$ we also have $\rho v\bot\,\reals$ if $\ga=\gan$.
\end{rem}

Related results and numerical tests for exterior domains can be found in e.g.
\cite{paulyrepinell,malimuzalevskiypaulyextell}.


\subsection{Eddy-Current (3D)} \label{subsec:EC}

Let $d=3$. The problem reads: Find the electric field $E\in\r$ such that
\begin{align}
\rot\mu^{-1}\rot E+\eps E&=J&
\textrm{in }&\om,\nonumber\\
n\times E&=0&
\textrm{on }&\gad,\label{eq:EC}\\
n\times\mu^{-1}\rot E&=0&
\textrm{on }&\gan,\nonumber
\end{align}
where
$$\r:=\set{\Phi\in\lt}{\rot\Phi\in\lt},\quad
\rz:=\set{\Phi\in\r}{\rot\Phi=0}.$$
We assume that the magnetic permeability $\mu$ and the electric permittivity $\eps$ 
are symmetric, real valued and uniformly positive definite matrices 
from $\li$. Of course, the extension to complex valued matrices is straight forward.
The electric current $J$ belongs to $\lt$. 
The dual variable for this problem is the magnetic field $H=\mu^{-1}\rot E\in\r$.
We define the Sobolev spaces
$$\rgad:=\ol{\cigad}^{\r},\quad
\rzgad:=\set{\Phi\in\rgad}{\rot\Phi=0}$$
and analogously $\rgan$ and $\rzgan$.
Moreover, we introduce the co-called Dirichlet-Neumann and Neumann-Dirichlet fields by
\begin{align*}
\harmdigadgan&:=\rzgad\cap\dzgan=\set{\Psi\in\rgad\cap\dgan}{\rot\Psi=0\,\wedge\,\div\Psi=0},\\
\harmdigangad&:=\rzgan\cap\dzgad=\set{\Psi\in\rgan\cap\dgad}{\rot\Psi=0\,\wedge\,\div\Psi=0},
\end{align*}
respectively. The following table shows the relation to the notation of Section \ref{sec:G}.
\begin{center}\begin{tabular}{c|c||c|c||c|c||c|c}
$\alphao$ & $\alphat$ & $\A$ & $\As$ & $\hio$ & $\hit$ & $D(\A)$ & $D(\As)$\\
\hline
$\eps$ & $\mu^{-1}$ & $\rot$ & $\rot$ & $\lt$ & $\lt$ & $\rgad$ & $\rgan$
\end{tabular}\end{center}
We note that indeed $D(\As)=\rgan$ holds for Lipschitz domains, see e.g. \cite{jochmanncompembmaxmixbc},
which is not trivial at all. The relation \eqref{partint} reads now
$$\forall\,\Phi\in\rgad\quad\forall\,\Psi\in\rgan\qquad
\scplt{\rot\Phi}{\Psi}=\scplt{\Phi}{\rot\Psi}.$$
Considering the norms we have
\begin{align*}
\norm{E}_{\r,\eps,\mu^{-1}}^2
&=\norm{E}_{\lt,\eps}^2
+\norm{\rot E}_{\lt,\mu^{-1}}^2,\\
\norm{H}_{\r,\eps^{-1},\mu}^2
&=\norm{H}_{\lt,\mu}^2
+\norm{\rot H}_{\lt,\eps^{-1}}^2,\\
\tnorm{(E,H)}^2
&=\norm{E}_{\r,\eps,\mu^{-1}}^2
+\norm{H}_{\r,\eps^{-1},\mu}^2.
\end{align*}
Now \eqref{eq:EC} reads: Find $E\in\rgad$ with $\mu^{-1}\rot E\in\rgan$ such that
$$\rot\mu^{-1}\rot E+\eps E=J.$$
In mixed formulation we have: Find $(E,H)\in\rgad\times\rgan$ such that
$$\rot H+\eps E=J,\quad\mu^{-1}\rot E=H.$$
The primal and dual variational problems are: Find $(E,H)\in\rgad\times\rgan$ such that
\begin{align*}
\forall\,\Phi&\in\rgad&
\scp{\rot E}{\rot\Phi}_{\lt,\mu^{-1}}+\scp{E}{\Phi}_{\lt,\eps}
&=\scplt{J}{\Phi},\\
\forall\,\Psi&\in\rgan&
\scp{\rot H}{\rot\Psi}_{\lt,\eps^{-1}}+\scp{H}{\Psi}_{\lt,\mu}
&=\scp{J}{\rot\Psi}_{\lt,\eps^{-1}}.
\end{align*}

\begin{theo} 
\label{thm:EC}
For any approximation $(\Et,\Ht)\in\rgad\times\rgan$ 
$$\tnorm{(E,H)-(\Et,\Ht)}^2
=\Mec(\Et,\Ht),\quad
\frac{\tnorm{(E,H)-(\Et,\Ht)}^2}{\tnorm{(E,H)}^2}
=\frac{\Mec(\Et,\Ht)}{\norm{J}_{\lt,\eps^{-1}}^2}$$
hold, where $\Mec(\Et,\Ht)=\norm{J-\eps\Et-\rot\Ht}_{\lt,\eps^{-1}}^2
+\norm{\Ht-\mu^{-1}\rot\Et}_{\lt,\mu}^2$.
\end{theo}

\begin{rem}
\label{rem:EC}
We note $\norm{E}_{\r,\eps,\mu^{-1}}\leq\norm{J}_{\lt,\eps^{-1}}$ 
and $\norm{H}_{\r,\eps^{-1},\mu}\leq\norm{J}_{\lt,\eps^{-1}}$ and indeed
$$\tnorm{(E,H)}=\norm{J}_{\lt,\eps^{-1}}.$$
The solution operator $L:\lt\to\rgad\times\rgan;f\mapsto(E,H)$ is an isometry, i.e. $\norm{L}=1$.
\end{rem}

\begin{cor}
\label{cor:EC}
Theorem \ref{thm:EC} provides the well known a posteriori error estimates 
for the primal and dual problems. 
\begin{itemize}
\item[\bf(i)]
For any $\Et\in\rgad$ it holds
$\ds\norm{E-\Et}_{\r,\eps,\mu^{-1}}^2
=\min_{\Psi\in\rgan}\Mec(\Et,\Psi)
=\Mec(\Et,H)$.
\item[\bf(ii)]
For any $\Ht\in\rgan$ it holds
$\ds\norm{H-\Ht}_{\r,\eps^{-1},\mu}^2
=\min_{\Phi\in\rgad}\Mec(\Phi,\Ht)
=\Mec(E,\Ht)$.
\end{itemize}
\end{cor}

\begin{rem}
We have $H=\mu^{-1}\rot E\in\rgan\cap\,\mu^{-1}\rot\rgad$ and $E$ and $(E,H)$ 
solve the strong and mixed formulation, respectively.
Moreover, $\rot H-J\in\eps\rgad$
with $\rot\eps^{-1}(\rot H-J)=-\mu H$ 
belonging to $\rot\rgad=\dzgad\cap\,\harmdigangad^{\bot}$.
Hence, for $J\in\eps\r$ we have $\rot H\in\eps\r$ and therefore
the strong and mixed formulations of the dual problem
\begin{align*}
\rot\eps^{-1}\rot H+\mu H&=\rot\eps^{-1}J&
&&
\textrm{in }&\om,\\
\rot D+\mu H&=\rot\eps^{-1}J,&
\eps^{-1}\rot H&=D&
\textrm{in }&\om
\intertext{hold, which are completed by the equations}
n\times\eps^{-1}(\rot H-J)&=0&
&&
\textrm{on }&\gad,\\
n\times H&=0&
&&
\textrm{on }&\gan,\\
\div\mu H&=0&
&&
\textrm{in }&\om,\\
n\cdot\mu H&=0&
&&
\textrm{on }&\gad,\\
\mu H&\,\,\bot\,\,\harmdigangad.
\end{align*}
Of course, $n\times D=n\times\eps^{-1}J$ on $\gad$
and by $\eps D\in\rot\rgan$ we also have
$\div\eps D=0$ in $\om$ and $n\cdot\eps D=0$ on $\gan$
as well as $\eps D\bot\harmdigadgan$.
\end{rem}

Earlier results for eddy current and static Maxwell problems can be found in 
\cite{anjammalimuzalevskiyneittaanmakirepinmaxtype,paulyrepinmaxst}.


\subsection{Eddy-Current (2D)} \label{subsec:EC2D}

Let $d=2$. We just indicate the changes compared to the latter section.
First, we have to understand the double $\rot$ as $\na^\bot\rot$, where
$$\rot E:=\div\R E=\p_1E_2-\p_2E_1,\quad
\na^\bot H:=\R\na H=\begin{bmatrix}\p_2H\\-\p_1H\end{bmatrix},\quad
\R:=\begin{bmatrix}0&1\\-1&0\end{bmatrix}$$
and $E\in\r$ is a vector field and $H\in\ho$ a scalar function. 
In the literature, the operator $\na^\bot$ is often called 
co-gradient or vector rotation $\vec{\rot}$ as well. 
Also $\mu$ is scalar. \eqref{eq:EC} reads: Find the electric field $E\in\r$ such that
\begin{align*}
\na^{\bot}\mu^{-1}\rot E+\eps E&=J&
\textrm{in }&\om,\\
n\times E&=0&
\textrm{on }&\gad,\\
\mu^{-1}\rot E&=0&
\textrm{on }&\gan.
\end{align*}
We have
\begin{center}\begin{tabular}{c|c||c|c||c|c||c|c}
$\alphao$ & $\alphat$ & $\A$ & $\As$ & $\hio$ & $\hit$ & $D(\A)$ & $D(\As)$\\
\hline
$\eps$ & $\mu^{-1}$ & $\rot$ & $\na^{\bot}$ & $\lt$ & $\lt$ & $\rgad$ & $\hogan$
\end{tabular}\end{center}
and \eqref{partint} turns to
$$\forall\,\Phi\in\rgad\quad\forall\,\psi\in\hogan\qquad
\scplt{\rot\Phi}{\psi}=\scplt{\Phi}{\na^{\bot}\psi}.$$
The norm for $H$ is
$$\norm{H}_{\ho,\eps^{-1},\mu}^2
=\norm{H}_{\lt,\mu}^2
+\norm{\na^{\bot}H}_{\lt,\eps^{-1}}^2.$$
The strong formulation of the problem is: 
Find $E\in\rgad$ with $\mu^{-1}\rot E\in\hogan$ such that
$$\na^{\bot}\mu^{-1}\rot E+\eps E=J.$$
The mixed formulation is: Find $(E,H)\in\rgad\times\hogan$ such that
$$\na^{\bot}H+\eps E=J,\quad\mu^{-1}\rot E=H.$$
The primal and dual variational problems are: Find $(E,H)\in\rgad\times\hogan$ such that
\begin{align*}
\forall\,\Phi&\in\rgad&
\scp{\rot E}{\rot\Phi}_{\lt,\mu^{-1}}+\scp{E}{\Phi}_{\lt,\eps}
&=\scplt{J}{\Phi},\\
\forall\,\psi&\in\hogan&
\scp{\na^{\bot}H}{\na^{\bot}\psi}_{\lt,\eps^{-1}}+\scp{H}{\psi}_{\lt,\mu}
&=\scp{J}{\na^{\bot}\psi}_{\lt,\eps^{-1}}.
\end{align*}

Theorem \ref{thm:EC} reads:

\begin{theo} 
\label{thm:EC2D}
For any approximation $(\Et,\Ht)\in\rgad\times\hogan$ 
$$\tnorm{(E,H)-(\Et,\Ht)}^2
=\Mec(\Et,\Ht),\quad
\frac{\tnorm{(E,H)-(\Et,\Ht)}^2}{\tnorm{(E,H)}^2}
=\frac{\Mec(\Et,\Ht)}{\norm{J}_{\lt,\eps^{-1}}^2}$$
hold, where $\Mec(\Et,\Ht)=\norm{J-\eps\Et-\na^{\bot}\Ht}_{\lt,\eps^{-1}}^2
+\norm{\Ht-\mu^{-1}\rot\Et}_{\lt,\mu}^2$.
\end{theo} 

\begin{rem}
\label{rem:EC2D}
We note $\norm{E}_{\r,\eps,\mu^{-1}}\leq\norm{J}_{\lt,\eps^{-1}}$ 
and $\norm{H}_{\ho,\eps^{-1},\mu}\leq\norm{J}_{\lt,\eps^{-1}}$ and indeed
$$\tnorm{(E,H)}=\norm{J}_{\lt,\eps^{-1}}.$$
The solution operator $L:\lt\to\rgad\times\hogan;f\mapsto(E,H)$ is an isometry, i.e. $\norm{L}=1$.
\end{rem}

\begin{cor}
\label{cor:EC2D}
Theorem \ref{thm:EC} provides the well known a posteriori error estimates 
for the primal and dual problems. 
\begin{itemize}
\item[\bf(i)]
For any $\Et\in\rgad$ it holds
$\ds\norm{E-\Et}_{\r,\eps,\mu^{-1}}^2
=\min_{\psi\in\hogan}\Mec(\Et,\psi)
=\Mec(\Et,H)$.
\item[\bf(ii)]
For any $\Ht\in\hogan$ it holds
$\ds\norm{H-\Ht}_{\ho,\eps^{-1},\mu}^2
=\min_{\Phi\in\rgad}\Mec(\Phi,\Ht)
=\Mec(E,\Ht)$.
\end{itemize}
\end{cor}

\begin{rem}
We have again $H=\mu^{-1}\rot E\in\hogan\cap\,\mu^{-1}\rot\rgad$ and
as in the 3D case $E$ and $(E,H)$ 
solve the strong and mixed formulation, respectively.
Moreover, $\na^{\bot}H-J\in\eps\rgad$
with $\rot\eps^{-1}(\na^{\bot}H-J)=-\mu H$.
Hence, for $J\in\eps\r$ we have $\na^{\bot}H\in\eps\r$ and therefore
the strong and mixed formulations of the dual problem
\begin{align*}
\rot\eps^{-1}\na^{\bot}H+\mu H&=\rot\eps^{-1}J&
&&
\textrm{in }&\om,\\
\rot D+\mu H&=\rot\eps^{-1}J,&
\eps^{-1}\na^{\bot}H&=D&
\textrm{in }&\om
\intertext{hold, which are completed by the equations}
n\times\eps^{-1}(\na^{\bot}H-J)&=0&
&&
\textrm{on }&\gad,\\
H&=0&
&&
\textrm{on }&\gan,\\
\mu H&\,\,\bot\,\,\reals\quad(\text{if }\gad=\ga).
\end{align*}
Of course, $n\times D=n\times\eps^{-1}J$ on $\gad$
and by $\eps D\in\na^{\bot}\hogan$ we also have
$\div\eps D=0$ in $\om$ and $n\cdot\eps D=0$ on $\gan$
as well as $\eps D\bot\harmdigadgan$.
\end{rem}


\subsection{Linear Elasticity} \label{subsec:LE}

Find the displacement vector field $u\in\ho$ such that
\begin{align}
-\Div\L\nas u+\rho\,u&=f&
\textrm{in }&\om,\nonumber\\
u&=0&
\textrm{on }&\gad,\label{eq:LE}\\
n\cdot\L\nas u&=0&
\textrm{on }&\gan.\nonumber
\end{align}
Here $\nas$ is the symmetric part of the gradient\footnote{Here, 
as usual in elasticity the gradient $\na u$
is to be understood as the Jacobian of the vector field $u$.}
$$\nas u:=\sym\na u=\frac{1}{2}\big(\na u+(\na u)^\top\big),$$
where ${}^\top$ denotes the transpose. 
$\nas u$, often denoted by $\eps(u)$, is also called the infinitesimal strain tensor.
The fourth order stiffness tensor of elastic moduli $\L\in\li$,
mapping symmetric matrices to symmetric matrices point-wise,
and the second order tensor (quadratic matrix) of reaction $\rho$ 
are assumed to be symmetric, real valued and uniformly positive definite.
The vector field $f$ (body force) belongs to $\lt$ and
the dual variable for this problem is the Cauchy stress tensor $\sigma=\L\nas u\in\d$,
where the application of $\Div$ to $\sigma$ and the notation $\sigma\in\d$ 
is to be understood row-wise as the usual divergence $\div$.
We note that the first equation can also be written as
$$-\Divs\L\nas u+\rho\,u=f,\quad\Divs:=\Div\sym.$$
We have:
\begin{center}\begin{tabular}{c|c||c|c||c|c||c|c}
$\alphao$ & $\alphat$ & $\A$ & $\As$ & $\hio$ & $\hit$ & $D(\A)$ & $D(\As)$\\
\hline
$\rho$ & $\L$ & $\nas$ & $-\Divs$ & $\lt$ & $\lt$ & $\hogad$ & $\sym^{-1}\dgan$
\end{tabular}\end{center}
The notation $\sigma\in\sym^{-1}\dgan$ means $\sym\sigma\in\dgan$.
More precisely, $\psi\in D(\As)$ if and only if
$$\forall\,\varphi\in D(\A)=\hogad\qquad
\scplt{\nas\varphi}{\psi}=\scplt{\varphi}{\As\psi}.$$
Since $\scplt{\nas\varphi}{\psi}=\scplt{\na\varphi}{\sym\psi}$
we see that this holds if and only if $\sym\psi\in\dgan$ and $\As\psi=-\Div\sym\psi$.
Equation \eqref{partint} turns into
$$\forall\,\varphi\in\hogad\quad\forall\,\psi\in\sym^{-1}\dgan
\qquad\scplt{\nas\varphi}{\psi}=-\scplt{\varphi}{\Divs\psi}.$$
For the norms we have
\begin{align*}
\norm{u}_{\ho,\rho,\L}^2
&=\norm{u}_{\lt,\rho}^2
+\norm{\nas u}_{\lt,\L}^2,\\
\norm{\sigma}_{\sym^{-1}\d,\rho^{-1},\L^{-1}}^2
&=\norm{\sigma}_{\lt,\L^{-1}}^2
+\norm{\Divs\sigma}_{\lt,\rho^{-1}}^2,\\
\tnorm{(u,\sigma)}^2
&=\norm{u}_{\ho,\rho,\L}^2
+\norm{\sigma}_{\sym^{-1}\d,\rho^{-1},\L^{-1}}^2.
\end{align*}
Now \eqref{eq:LE} reads: Find $u\in\hogad$ with $\sym\L\nas u=\L\nas u\in\dgan$ such that
$$-\Div\L\nas u+\rho\,u=f.$$
In mixed formulation we have: Find $(u,\sigma)\in\hogad\times\dgan$ such that
$$-\Div\sigma+\rho\,u=f,\quad\L\nas u=\sigma.$$
Note that then $\sigma$ is automatically symmetric.
The primal and dual variational problems are: Find $(u,\sigma)\in\hogad\times\sym^{-1}\dgan$ such that
\begin{align*}
\forall\,\varphi&\in\hogad&
\scp{\nas u}{\nas\varphi}_{\lt,\L}+\scp{u}{\varphi}_{\lt,\rho}
&=\scplt{f}{\varphi},\\
\forall\,\psi&\in\sym^{-1}\dgan&
\scp{\Divs\sigma}{\Divs\psi}_{\lt,\rho^{-1}}+\scp{\sigma}{\psi}_{\lt,\L^{-1}}
&=-\scp{f}{\Divs\psi}_{\lt,\rho^{-1}}.
\intertext{Since $\sigma\in\dgan$ must be symmetric, we can formulate the dual problem also as}
\forall\,\psi&\in\dgan,\;\psi\text{ symmetric}&
\scp{\Div\sigma}{\Div\psi}_{\lt,\rho^{-1}}+\scp{\sigma}{\psi}_{\lt,\L^{-1}}
&=-\scp{f}{\Div\psi}_{\lt,\rho^{-1}}.
\end{align*}
Then, the norms reduce to
$$\tnorm{(u,\sigma)}^2
=\norm{u}_{\ho,\rho,\L}^2
+\norm{\sigma}_{\d,\rho^{-1},\L^{-1}}^2,\quad
\norm{\sigma}_{\d,\rho^{-1},\L^{-1}}^2
=\norm{\sigma}_{\lt,\L^{-1}}^2
+\norm{\Div\sigma}_{\lt,\rho^{-1}}^2.$$

\begin{theo} 
\label{thm:LE}
For any approximation $(\ut,\tilde{\sigma})\in\hogad\times\sym^{-1}\dgan$ 
\begin{align}
\label{linelatheoeq}
\tnorm{(u,\sigma)-(\ut,\tilde{\sigma})}^2
=\Mle(\ut,\tilde{\sigma}),\quad
\frac{\tnorm{(u,\sigma)-(\ut,\tilde{\sigma})}^2}{\tnorm{(u,\sigma)}^2}
=\frac{\Mle(\ut,\tilde{\sigma})}{\norm{f}_{\lt,\rho^{-1}}^2}
\end{align}
hold, where $\Mle(\ut,\tilde{\sigma})=\norm{f-\rho\ut+\Divs\tilde{\sigma}}_{\lt,\rho^{-1}}^2
+\norm{\tilde{\sigma}-\L\nas\ut}_{\lt,\L^{-1}}^2$.
Moreover, since $\sigma$ is automatically symmetric we have \eqref{linelatheoeq} for all
$(\ut,\tilde{\sigma})\in\hogad\times\dgan$ with $\tilde{\sigma}$ symmetric
and the right hand side simplifies to
$\Mle(\ut,\tilde{\sigma})=\norm{f-\rho\ut+\Div\tilde{\sigma}}_{\lt,\rho^{-1}}^2
+\norm{\tilde{\sigma}-\L\nas\ut}_{\lt,\L^{-1}}^2$.
\end{theo}

\begin{rem}
We note $\norm{u}_{\ho,\rho,\L}\leq\norm{f}_{\lt,\rho^{-1}}$ 
and $\norm{\sigma}_{\d,\rho^{-1},\L^{-1}}\leq\norm{f}_{\lt,\rho^{-1}}$ and indeed
$$\tnorm{(u,\sigma)}=\norm{f}_{\lt,\rho^{-1}}.$$
The solution operator $L:\lt\to\hogad\times\dgan;f\mapsto(u,\sigma)$ is an isometry, i.e. $\norm{L}=1$.
\end{rem}

\begin{cor}
Theorem \ref{thm:LE} provides the well known a posteriori error estimates 
for the primal and dual problems. 
\begin{itemize}
\item[\bf(i)]
For any $\ut\in\hogad$ it holds
$\ds\norm{u-\ut}_{\ho,\rho,\Amat}^2
=\min_{\psi\in\sym^{-1}\dgan}\Mle(\ut,\psi)
=\Mle(\ut,\sigma)$.
\item[\bf(ii)]
For any $\tilde{\sigma}\in\sym^{-1}\dgan$ it holds
$\ds\norm{\sigma-\tilde{\sigma}}_{\sym^{-1}\d,\rho^{-1},\Amat^{-1}}^2
=\min_{\varphi\in\hogad}\Mle(\varphi,\tilde{\sigma})
=\Mle(u,\tilde{\sigma})$.
\end{itemize}
If $\tilde{\sigma}$ and $\psi$ are already symmetric we can skip the $\sym^{-1}$
and replace $\Divs$ by $\Div$.
\end{cor}

\begin{rem}
We have $\sigma=\L\nas u\in\dgan\cap\,\L\nas\hogad$ is symmetric with $\Divs\sigma=\Div\sigma$
and $u$ and $(u,\sigma)$ solve the strong and mixed formulation, respectively.
Moreover, $\Div\sigma+f\in\rho\hogad$
with $\nas\rho^{-1}(\Div\sigma+f)=\L^{-1}\sigma\in\nas\hogad$.
Hence, for $f\in\rho\ho$ we have $\Div\sigma\in\rho\ho$ and therefore
strong and mixed formulations of the dual problem hold, i.e.,
\begin{align*}
-\nas\rho^{-1}\Div\sigma+\L^{-1}\sigma&=\nas\rho^{-1}f&
&&
\textrm{in }&\om,\\
\nas v+\L^{-1}\sigma&=\nas\rho^{-1}f,&
-\rho^{-1}\Div\sigma&=v&
\textrm{in }&\om.
\end{align*}
\end{rem}


\subsection{Generalized Reaction-Diffusion, Linear Accoustics and Eddy-Current} \label{subsec:diff}

Let $\om$ be a $d$-dimensional smooth Riemannian manifold 
with compact Lipschitz boundary $\ga$. If $\om$ is unbounded, 
we assume that outside of some compact set, $\om$ is isomorphic 
to the exterior unit domain $\set{x\in\reals^{d}}{|x|>1}$.
Moreover, let $\gad$ be an open subset of $\ga$ and $\gan:=\ga\setminus\ol{\gad}$ its complement.
The problem reads:
For $f\in\ltq$ find the differential form potential ($q$-form) $u\in\dq$, such that
\begin{align}
-\cd\Amat\ed u+\rho\,u&=f&\textrm{in }&\om,\nonumber\\
\ttrgad u&=0&\textrm{on }&\gad,\label{eq:diff}\\
\ntrgan \Amat\ed u&=0&\textrm{on }&\gan\nonumber.
\end{align}
Here, $\ed$ denotes exterior derivative, $\cd=\pm*\ed*$ the co-derivative
and $\ttrgad$ resp. $\ntrgan$ the restrictions of the tangential resp. normal
traces $\ttrga$ resp. $\ntrga$ to the proper subspaces.
We also introduce the Sobolev spaces
$$\dq:=\set{\varphi\in\ltq}{\ed\varphi\in\ltqpo},\quad
\deq:=\set{\psi\in\ltq}{\cd\psi\in\ltqmo}$$
and $\dqgad:=\ol{\ciqgad}^{\dq}$, $\deqgan:=\ol{\ciqgan}^{\deq}$,
where $\ciqgad$ resp. $\ciqgan$ are smooth test $q$-forms
having supports bounded away from $\gad$ resp. $\gan$.
Moreover, $\ltq$ denotes the Lebesgue space of all square integrable $q$-forms on $\om$
equipped with the inner or scalar product 
$$\scpltq{u}{\varphi}:=\int_{\om}u\wedge*\varphi$$
and corresponding norm $\normltq{\,\cdot\,}$. Of course, $\dq$ and $\deq$
are equipped with the respective graph norms, making them Hilbert spaces.
Finally, $\rho$ and $\Amat$ denote linear, symmetric, real valued, bounded 
and uniformly positive definite transformations on $q$- resp. $(q+1)$-forms. 
It is again straight forward to discuss complex valued transformations.
We also need the spaces
$$\dqz:=\set{\varphi\in\dq}{\ed\varphi=0},\quad
\dqzgad:=\set{\varphi\in\dqgad}{\ed\varphi=0}$$
and the corresponding spaces for the co-derivative as well as 
the space of harmonic Dirichlet-Neumann forms
$$\harmdiqgadgan:=\dqzgad\cap\deqzgan.$$
The dual variable for this problem is the `flux' $p=\Amat\ed u\in\deqpo$.
The next table shows the relation to the notations of Section \ref{sec:G}.
\begin{center}\begin{tabular}{c|c||c|c||c|c||c|c}
$\alphao$ & $\alphat$ & $\A$ & $\As$ & $\hio$ & $\hit$ & $D(\A)$ & $D(\As)$\\
\hline
$\rho$ & $\Amat$ & $\ed$ & $-\cd$ & $\ltq$ & $\ltqpo$ & $\dqgad$ & $\deqpogan$
\end{tabular}\end{center}
Also here indeed $D(\As)=\deqpogan$ holds, see 
e.g. \cite{goldshteinmitreairinamariushodgedecomixedbc,jakabmitreairinamariusfinensolhodgedeco,kuhndiss}. 
The relation \eqref{partint} turns into
$$\forall\,\varphi\in\dqgad\quad\forall\,\psi\in\deqpogan\qquad
\scpltqpo{\ed\varphi}{\psi}=-\scpltq{\varphi}{\cd\psi}.$$
Considering the norms we have
\begin{align*}
\norm{u}_{\dq,\rho,\Amat}^2
&=\norm{u}_{\ltq,\rho}^2
+\norm{\ed u}_{\ltqpo,\Amat}^2,\\
\norm{p}_{\deqpo,\rho^{-1},\Amat^{-1}}^2
&=\norm{p}_{\ltqpo,\Amat^{-1}}^2
+\norm{\cd p}_{\ltq,\rho^{-1}}^2,\\
\tnorm{(u,p)}^2
&=\norm{u}_{\dq,\rho,\Amat}^2
+\norm{p}_{\deqpo,\rho^{-1},\Amat^{-1}}^2.
\end{align*}
Now \eqref{eq:diff} reads: Find $u\in\dqgad$ with $\Amat\ed u\in\deqpogan$ such that
$$-\cd\Amat\ed u+\rho\,u=f.$$
In mixed formulation we have: Find $(u,p)\in\dqgad\times\deqpogan$ such that
$$-\cd p+\rho\,u=f,\quad\Amat\ed u=p.$$
The primal and dual variational problems are: Find $(u,p)\in\dqgad\times\deqpogan$ such that
\begin{align*}
\forall\,\varphi&\in\dqgad&
\scp{\ed u}{\ed\varphi}_{\ltqpo,\Amat}+\scp{u}{\varphi}_{\ltq,\rho}
&=\scpltq{f}{\varphi},\\
\forall\,\psi&\in\deqpogan&
\scp{\cd p}{\cd\psi}_{\ltq,\rho^{-1}}+\scp{p}{\psi}_{\ltqpo,\Amat^{-1}}
&=-\scp{f}{\cd\psi}_{\ltq,\rho^{-1}}.
\end{align*}

\begin{theo} 
\label{thm:diff}
For any approximation $(\ut,\pt)\in\dqgad\times\deqpogan$ 
$$\tnorm{(u,p)-(\ut,\pt)}^2
=\Mdiff(\ut,\pt),\quad
\frac{\tnorm{(u,p)-(\ut,\pt)}^2}{\tnorm{(u,p)}^2}
=\frac{\Mdiff(\ut,\pt)}{\norm{f}_{\ltq,\rho^{-1}}^2}$$
hold, where $\Mdiff(\ut,\pt)=\norm{f-\rho\ut+\cd\pt}_{\ltq,\rho^{-1}}^2
+\norm{\pt-\Amat\ed\ut}_{\ltqpo,\Amat^{-1}}^2$.
\end{theo}

\begin{rem}
We note $\norm{u}_{\dq,\rho,\Amat}\leq\norm{f}_{\ltq,\rho^{-1}}$ 
and $\norm{p}_{\deqpo,\rho^{-1},\Amat^{-1}}\leq\norm{f}_{\ltq,\rho^{-1}}$ and indeed
$$\tnorm{(u,p)}=\norm{f}_{\ltq,\rho^{-1}}.$$
The solution operator $L:\ltq\to\dqgad\times\deqpogan;f\mapsto(u,p)$ is an isometry, i.e. $\norm{L}=1$.
\end{rem}

\begin{cor}
Theorem \ref{thm:diff} provides the a posteriori error estimates 
for the primal and dual problems. 
\begin{itemize}
\item[\bf(i)]
For any $\ut\in\dqgad$ it holds
$\ds\norm{u-\ut}_{\dq,\rho,\Amat}^2
=\min_{\psi\in\deqpogan}\Mdiff(\ut,\psi)
=\Mdiff(\ut,p)$.
\item[\bf(ii)]
For any $\pt\in\deqpogan$ it holds
$\ds\norm{p-\pt}_{\deqpo,\rho^{-1},\Amat^{-1}}^2
=\min_{\varphi\in\dqgad}\Mdiff(\varphi,\pt)
=\Mdiff(u,\pt)$.
\end{itemize}
\end{cor}

We note that for $q=0$ we get back the reaction-diffusion problem from Section \ref{subsec:RD}
and for $d=3$ or $d=2$ and $q=1$ we obtain the eddy-current problems 
from Sections \ref{subsec:EC} and \ref{subsec:EC2D},
identifying $\om\subset\reals^{d}$ with a proper domain and
$0$-forms with functions and $1$- and $2$-forms with vector fields
by Riesz' representation theorem and Hodge's star operator.

\begin{rem}
It holds $p=\Amat\ed u\in\deqpogan\cap\,\Amat\ed\dqgad$ and $u$ and $(u,p)$ 
solve the strong and mixed formulations, respectively.
Moreover, $\cd p+f$ belongs to $\rho\dqgad$ and we see immediately 
$\ed\rho^{-1}(\cd p+f)=\Amat^{-1}p\in\ed\dqgad=\dqpozgad\cap\,(\harmdiqpogadgan)^{\bot}$.
Hence, for $f\in\rho\dq$ we have $\cd p\in\rho\dq$ and therefore
the strong and mixed formulations of the dual problem
\begin{align*}
-\ed\rho^{-1}\cd p+\Amat^{-1}p&=\ed\rho^{-1}f&
&&
\textrm{in }&\om,\\
\ed v+\Amat^{-1}p&=\ed\rho^{-1}f,&
-\rho^{-1}\cd p&=v&
\textrm{in }&\om
\intertext{hold, which are completed by the equations}
\ttrgad\rho^{-1}(\cd p+f)&=0&
&&
\textrm{on }&\gad,\\
\ttrgan p&=0&
&&
\textrm{on }&\gan,\\
\ed\Amat^{-1}p&=0&
&&
\textrm{in }&\om,\\
\ttrgad\Amat^{-1}p&=0&
&&
\textrm{on }&\gad,\\
\Amat^{-1}p&\,\,\bot\,\,\harmdiqpogadgan.
\end{align*}
Of course, there are also more equations for $v$
following from $\rho v\in\rho\dq\cap\cd\deqpogan$, e.g. $\cd\rho v=0$,
which we will not list here explicitly.
\end{rem}


\section{Inhomogeneous and More Boundary Conditions}
\label{sec:inhomobc}

In this section we will demonstrate that our error equalities also hold 
for Robin type boundary conditions, which means that our error equalities are true
for many commonly used boundary conditions. Moreover, we emphasize that we can also handle
inhomogeneous boundary conditions.
Since it is clear that this method works in the general setting as well
we will discuss it here just for the simple reaction-diffusion model problem from the introduction.

Let $\om$ be as in the latter section and now the boundary $\ga$ 
be decomposed into three disjoint parts $\gad$, $\gan$ and $\gar$.
The model problem is:
Find the scalar potential $u\in\ho$ such that
\begin{align*}
-\div\na u+u&=f&\textrm{in }&\om,\\
u&=g_1&\textrm{on }&\gad,\\
n\cdot\na u&=g_2&\textrm{on }&\gan,\\
n\cdot\na u+\gamma u&=g_3&\textrm{on }&\gar
\end{align*}
hold. Hence, on $\gad, \gan$ and $\gar$ we impose Dirichlet, Neumann 
and Robin type boundary conditions, respectively. 
In the Robin boundary condition, we assume that the coefficient $\gamma\ge\gamma_0>0$ 
belongs to $\li$. The dual variable for this problem 
is the flux $p:=\na u\in\d$.
Furthermore, as long as $\gar\neq\emptyset$ and
to avoid tricky discussions about traces 
and the corresponding $\hmoh$-spaces of $\ga$, $\gad, \gan$ and $\gar$,
which can be quite complicated, we assume for simplicity that
$u\in\htwo$. Then, $p\in\ho$ and all $g_{i}$ belong to $\lt$ even to $\hoh$ of $\ga$.
For the norms we simply have
$$\tnorm{(u,p)}^2
=\normho{u}^2
+\normd{p}^2.$$

\begin{theo} 
\label{thm:dnr}
For any approximation $(\ut,\pt)\in\htwo\times\ho$ with
$u-\ut\in\hogad$ and $p-\pt\in\dgan$ as well as
$n\cdot(p-\pt)+\gamma(u-\ut)=0$ on $\gar$
$$\tnorm{(u,p)-(\ut,\pt)}^2
+\norm{u-\ut}_{\lt(\gar),\gamma}^2
+\norm{n\cdot(p-\pt)}_{\lt(\gar),\gamma^{-1}}^2
=\Mmix(\ut,\pt)$$
holds with $\Mmix$ from Theorem \ref{thm:Imain}.
Moreover, $\norm{u-\ut}_{\lt(\gar),\gamma}=\norm{n\cdot(p-\pt)}_{\lt(\gar),\gamma^{-1}}$.
\end{theo}

\begin{proof}
Following Remark \ref{rem:Iproof} we have
\begin{align*}
\Mmix(\ut,\pt)
&=\ubr{\normho{u-\ut}^2
+\normd{p-\pt}^2}_{\ds=\tnorm{(u,p)-(\ut,\pt)}^2}
+2\scplt{\na(u-\ut)}{\pt-p}
+2\scplt{u-\ut}{\div(\pt-p)}.
\end{align*}
Moreover, since $n\cdot(\pt-p)$ and $u-\ut$ belong to $\lt(\ga)$ we have 
\begin{align*}
&\qquad\scplt{\na(u-\ut)}{\pt-p}
+\scplt{u-\ut}{\div(\pt-p)}\\
&=\scp{n\cdot(\pt-p)}{u-\ut}_{\lt(\ga)}
=\scp{n\cdot(\pt-p)}{u-\ut}_{\lt(\gar)}
=\scp{\gamma(u-\ut)}{u-\ut}_{\lt(\gar)}.
\end{align*}
As $\scp{\gamma(u-\ut)}{u-\ut}_{\lt(\gar)}=\scp{\gamma^{-1}n\cdot(p-\pt)}{n\cdot(p-\pt)}_{\lt(\gar)}$
we get the assertion.
\end{proof}

\begin{rem}
If all $g_{i}=0$, we can set $(\ut,\pt)=(0,0)$ and get
$$\tnorm{(u,p)}^2
+\norm{u}_{\lt(\gar),\gamma}^2
+\norm{n\cdot p}_{\lt(\gar),\gamma^{-1}}^2
=\normlt{f}^2,$$
which follows also directly from Remark \ref{thm:Imainrem} (ii'),
$p=\na u$ and $n\cdot p=-\gamma u$ on $\gar$ as well as
\begin{align*}
\normlt{f}^2
&=\normlt{\div p}^2
+\normlt{u}^2
-2\scplt{\div\na u}{u}\\
&=\normlt{\div p}^2
+\normlt{u}^2
+2\normlt{\na u}
-2\scp{n\cdot\na u}{u}_{\lt(\ga)}\\
&=\normlt{\div p}^2
+\normlt{u}^2
+2\normlt{\na u}
-2\ubr{\scp{n\cdot\na u}{u}_{\lt(\gar)}}_{\ds=-\norm{u}_{\lt(\gar),\gamma}^2}.
\end{align*}
Thus, in this case the assertion of Theorem \ref{thm:dnr} has a normalized counterpart as well.
\end{rem}

If $\gar=\emptyset$ we have a pure mixed Dirichlet and Neumann boundary. 

\begin{theo} 
\label{thm:dnrnorobin}
Let $\gar=\emptyset$.
For any approximation $(\ut,\pt)\in\ho\times\d$ with
$u-\ut\in\hogad$ and $p-\pt\in\dgan$
$$\tnorm{(u,p)-(\ut,\pt)}^2
=\Mmix(\ut,\pt)$$
holds with $\Mmix$ from Theorem \ref{thm:Imain}.
\end{theo}

\begin{cor}
Let $\gar=\emptyset$.
Theorem \ref{thm:dnrnorobin} provides the well known a posteriori error estimates 
for the primal and dual problems. 
\begin{itemize}
\item[\bf(i)]
For any $\ut\in\ho$ with $u-\ut\in\hogad$ it holds
$\ds\normho{u-\ut}^2
=\min_{\substack{\psi\in\d\\p-\psi\in\dgan}}\Mmix(\ut,\psi)
=\Mmix(\ut,p)$.
\item[\bf(ii)]
For any $\pt\in\d$ with $p-\pt\in\dgan$ it holds
$\ds\normd{p-\pt}^2
=\min_{\substack{\varphi\in\ho\\u-\varphi\in\hogad}}\Mmix(\varphi,\pt)
=\Mmix(u,\pt)$.
\end{itemize}
\end{cor}


\section{Numerical Examples} \label{sec:N}

In this section we show by some academic test cases 
the numerical performance of our error equalities.
All the calculations have been done using MATLAB, 
and the reported values in the tables have not been rounded, 
but are simply cut-offs of values reported by MATLAB. 
The main quantity of interest is the difference between the exact error 
and the value given by the majorant for a certain approximation $(\ut,\pt)$, i.e.,
$$\delta:=\big|\tnorm{(u,p)-(\ut,\pt)}-\sqrts{\M_{\cdots}(\ut,\pt)}\big|,$$
where the test problems are either from the reaction-diffusion problems from subsection \ref{subsec:RD} 
or from the eddy-current problems from subsections \ref{subsec:EC} and \ref{subsec:EC2D}. 
Where the finite element method (FEM) has been used, 
we have employed only linear triangular elements in 2D and linear tetrahedral elements in 3D. 
In all the examples below we calculated the approximations $\ut$ and $\pt$ (or $\Et$ and $\Ht$)
in the same mesh only for the sake of convenience. 
Using different meshes for the primal and dual approximations is allowed. 
We also used only regular meshes, but irregular meshes can be used as well. 
The only requirement is that the approximations must be conforming, 
meaning that they belong to the appropriate Sobolev spaces 
and fulfill the boundary conditions exactly.
All finite element solvers were implemented in the vectorized manner explained in \cite{rahmanvaldmanfastmatlab}.

\begin{ex}
\label{ex1}
We take the 3D-reaction-diffusion problem from Section \ref{subsec:RD}
and choose the unit cube $\om:=(0,1)^3$ with exact solution 
$$u(x):=\prod_{i=1}^{3}x_{i}(1-x_{i}),$$
where $u$ satisfies the zero Dirichlet boundary conditions on the whole boundary, i.e.,
$\gad=\ga$ and $\gan=\emptyset$,
and the following data
$$\Amat(x):=\Amat:=\begin{bmatrix}1&0&0\\0&5&0\\0&0&10\end{bmatrix},\quad
\rho(x):=\begin{cases}
1 & \textrm{if} \quad 0 < x_1 < \nicefrac{1}{4} \\
10 & \textrm{if} \quad \nicefrac{1}{4} < x_1 < \nicefrac{3}{4} \\
25 & \textrm{if} \quad \nicefrac{3}{4} < x_1 < 1
\end{cases}.$$
This means that the approximation 
of the dual variable does not have any boundary condition. 
We calculated the approximation globally by solving 
the primal and dual problem
with standard linear Courant elements and linear Raviart-Thomas elements, respectively. 
We will denote this finite element approximation pair by $(\uh,\ph)$. 
The resulting linear systems were solved directly in MATLAB. 
The approximations were calculated in uniformly refined regular meshes, 
where the jumps in the reaction coefficient $\rho$ coincide with element boundaries. 
For each mesh we computed the exact combined error and the majorant $\Mrd(\uh,\ph)$. 
The results are displayed in Table \ref{tbl:Ex1}. 
The first column shows the number of elements $\Nelem$ of the mesh. 
The second and third column show the exact error and the value given by the majorant. 
The fourth column shows the difference $\delta$ between the exact error and the value given by the majorant.
\end{ex}

\begin{table} [htdp]
\caption{Example \ref{ex1} (3D-reaction-diffusion)}
\begin{center}
\begin{tabular}{r|c|c|c}
$\Nelem$ & $\tnorm{(u,p)-(\uh,\ph)}$ & $\sqrts{\Mrd(\uh,\ph)}$ & difference $\delta$ \\
\hline\hline
384    & 0.12803218100 & 0.12803218100 & 5.551115123e-17 \\
3072   & 0.06736516349 & 0.06736516349 & 4.163336342e-17 \\
24576  & 0.03433600867 & 0.03433600867 & 9.714451465e-17  \\
196608 & 0.01728806289 & 0.01728806289 & 3.469446952e-18
\end{tabular}
\end{center}
\label{tbl:Ex1}
\end{table}

\begin{ex}
\label{ex2}
This test is similar to the Example 1 except that the linear systems resulting 
from the finite element computations were not solved directly, but with an iterative method,
where the stopping tolerance was set to the crude value of $10^{-4}$.
The approximation pair obtained by this method is denoted by $(\uiter,\piter)$. 
No preconditioning was done. The iterative solver of the linear system of the dual problem converged 
only for the smallest mesh, and the error actually grows between the two last meshes. 
With this stopping tolerance this is expected
and was purposefully done so in order to obtain approximations 
which are relatively far from having the Galerkin orthogonality property. 
We did this test simply to demonstrate that Galerkin orthogonality 
is not a requirement for the equality to hold. 
The results are displayed in Table \ref{tbl:Ex2}.
\end{ex}

\begin{table} [htdp]
\caption{Example \ref{ex2} (3D-reaction-diffusion)}
\begin{center}
\begin{tabular}{r|c|c|c}
$\Nelem$ & $\tnorm{(u,p)-(\uiter,\piter)}$ & $\sqrts{\Mrd(\uiter,\piter)}$ & difference $\delta$ \\
\hline\hline
384    & 0.12803483290 & 0.12803483290 & 2.775557562e-17 \\
3072   & 0.06868358511 & 0.06868358511 & 6.938893904e-17 \\
24576  & 0.05294561599 & 0.05294561599 & 6.245004514e-17 \\
196608 & 0.09166231565 & 0.09166231565 & 9.714451465e-17
\end{tabular}
\end{center}
\label{tbl:Ex2}
\end{table}

\begin{ex}
\label{ex3}
We ran the problem data of Example \ref{ex1} with subsequently refined regular meshes, 
where the approximation of the primal variable $\uh$ was again obtained by the linear Courant finite elements. 
The resulting linear system was solved directly. 
The approximation of the dual variable was calculated by averaging the values 
$\Amat\na\uh$ to the nodes of the mesh. This procedure is often called the gradient averaging method
and we will denote the resulting function by $\pav$. 
The results can be seen in Table \ref{tbl:Ex3}.
\end{ex}

\begin{table} [htdp]
\caption{Example \ref{ex3} (3D-reaction-diffusion)}
\begin{center}
\begin{tabular}{r|c|c|c}
$\Nelem$ & $\tnorm{(u,p)-(\uh,\pav)}$ & $\sqrts{\Mrd(\uh,\pav)}$ & difference $\delta$ \\
\hline\hline
384    & 0.2698605861 & 0.2698605861 & 0 \\
3072   & 0.2285323585 & 0.2285323585 & 0 \\
24576  & 0.1831121412 & 0.1831121412 & 6.106226635e-16 \\
196608 & 0.1333268308 & 0.1333268308 & 1.693090113e-15
\end{tabular}
\end{center}
\label{tbl:Ex3}
\end{table}

\begin{ex}
\label{ex4}
We take the 2D-eddy-current problem from Section \ref{subsec:EC2D} 
and choose the unit square $\om:=(0,1)^2$ 
with $\eps=\id$ and $\mu=1$. We split the domain in two parts
$\om_1:=\set{x\in\om}{x_1>x_2}$ and $\om_2=\om\setminus\ol{\om_1}$ 
in order to define the following discontinuous solution
$$E|_{\om_1}(x):=
\begin{bmatrix}
\sin(2\pi x_1)+2\pi\cos(2\pi x_1)(x_1-x_2)\\
\sin\big((x_1-x_2)^2(x_1-1)^2x_2\big)-\sin(2\pi x_1)
\end{bmatrix},\quad
E|_{\om_2}(x):=0.$$
Note that indeed $E\in\r\setminus\ho$ and $\rot E\in\ho$ with
$$\rot E|_{\om_1}(x)=2x_{2}(x_1-x_2)(x_1-1)(2x_{1}-x_{2}-1)\cos(2\pi x_1).$$
We set zero Neumann boundary conditions on the whole boundary, i.e.,
$\gad=\emptyset$ and $\gan=\ga$. 
The exact solution and its rotation is visualized in Figure \ref{fig:ex4}.
We calculated the approximation globally by solving 
the primal and dual problem with linear N\'ed\'elec elements and linear Courant elements, respectively. 
This finite element approximation pair will be denoted by $(\Eh,\Hh)$. 
The resulting linear systems were solved directly. 
The approximations were calculated in uniformly refined regular meshes, 
where the jumps in the exact solution and in the right hand side $J$ coincide with element boundaries. 
For each mesh we calculated the exact combined error and the majorant $\Mec(\Eh,\Hh)$. 
The results are displayed in Table \ref{tbl:Ex4}.
\end{ex}

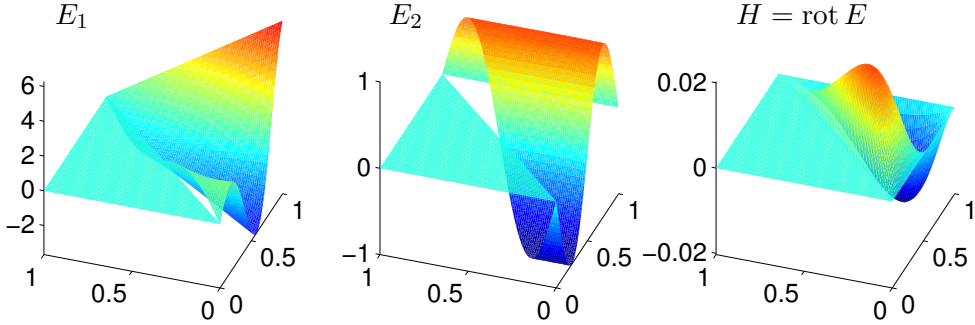
\begin{figure}[!h]
\centering
\begin{picture}(450,140)
\put(0,0){\includegraphics[scale=0.6]{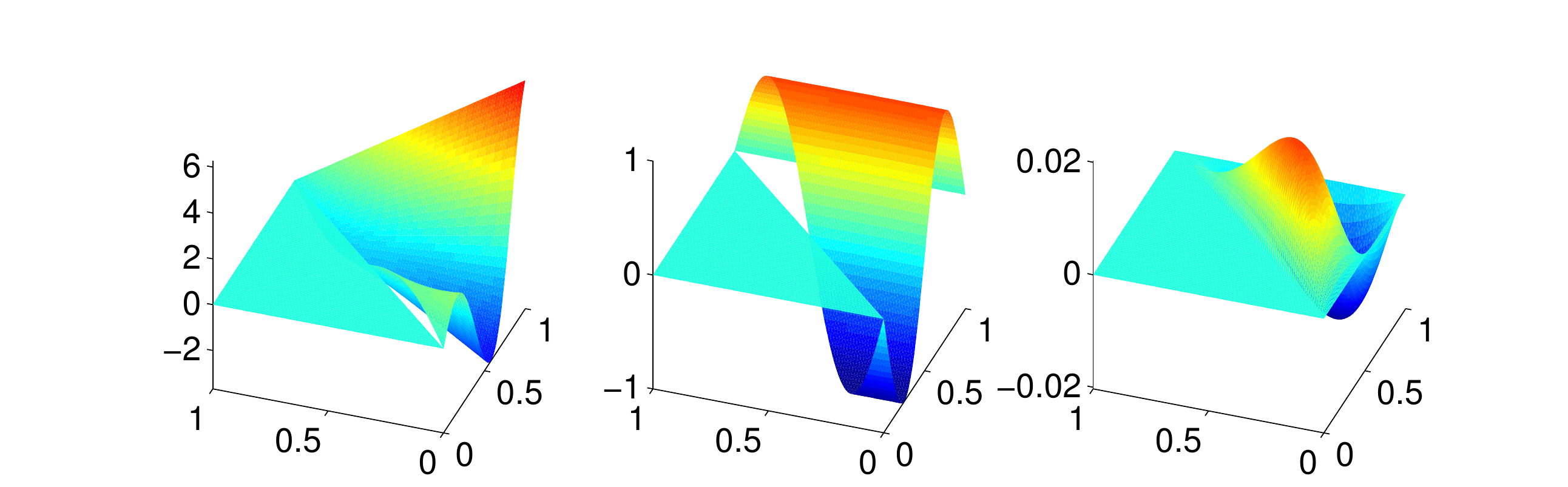}}
\put(65,115){\small$E_1$}
\put(190,115){\small$E_2$}
\put(320,115){\small$H=\rot E$}
\end{picture}
\caption{The two components of the exact solution $E$ and its rotation $H$ of Example \ref{ex4}.}
\label{fig:ex4}
\end{figure}

\begin{table} [htdp]
\caption{Example \ref{ex4} (2D-eddy-current)}
\begin{center}
\begin{tabular}{r|c|c|c}
$\Nelem$ & $\tnorm{(E,H)-(\Eh,\Hh)}$ & $\sqrts{\Mec(\Eh,\Hh)}$ & difference $\delta$ \\
\hline\hline
800    & 0.151485078300 & 0.151485078300 & 2.220446049e-16 \\
3200   & 0.075877018950 & 0.075877018950 & 0 \\
12800  & 0.037956449900 & 0.037956449900 & 7.632783294e-17 \\
51200  & 0.018980590110 & 0.018980590110 & 6.938893904e-17 \\
204800 & 0.009490605462 & 0.009490605462 & 2.602085214e-17
\end{tabular}
\end{center}
\label{tbl:Ex4}
\end{table}

\begin{ex}
\label{ex5}
We take the 3D-eddy-current problem from Section \ref{subsec:EC} 
and choose the unit cube $\om:=(0,1)^3$ 
with $\eps=\mu=\id$. Again we split the domain in two parts
$\om_1:=\set{x\in\om}{x_1>x_2}$ and $\om_2=\om\setminus\ol{\om_1}$ 
in order to define the following discontinuous solution
$$E(x):=
\chi_{\om_{1}}(x)
\begin{bmatrix}
\sin(2\pi x_1)+2\pi\cos(2\pi x_1)(x_1-x_2)\\
\sin\big((x_1-x_2)^2(x_1-1)^2x_2\big)-\sin(2\pi x_1)\\
0
\end{bmatrix}
+\xi(x)\begin{bmatrix}0\\0\\1\end{bmatrix},\quad
\xi(x):=\prod_{i=1}^{3}x_{i}^2(1-x_{i})^2.$$
Thus, we extended the discontinuous vector field of Example \ref{ex4} 
by zero in the third component and added a smooth bubble in the third component.
Hence, $E\in\r\setminus\ho$ and $\rot E\in\r$ with
$$\rot E(x)=
\chi_{\om_{1}}(x)\big(2x_{2}(x_1-x_2)(x_1-1)(2x_{1}-x_{2}-1)\cos(2\pi x_1)\big)
\begin{bmatrix}0\\0\\1\end{bmatrix}
+\begin{bmatrix}\p_{2}\xi\\-\p_{1}\xi\\0\end{bmatrix}(x).$$
Note that even $\rot E\in\ho$ holds.
We set zero Neumann boundary conditions on the whole boundary, i.e.,
$\gad=\emptyset$ and $\gan=\ga$. 
We calculated the approximation globally by solving 
the primal and dual problem with linear N\'ed\'elec elements.
This finite element approximation pair will be denoted by $(\Eh,\Hh)$.
The resulting linear systems were solved directly.
The approximations were calculated in uniformly refined regular meshes, 
where the jumps in the exact solution and in the right hand side $J$ coincide with element boundaries.
For each mesh we calculated the exact combined error and the majorant $\Mec(\Eh,\Hh)$. 
The results are displayed in Table \ref{tbl:Ex5}.
\end{ex}

\begin{table} [htdp]
\caption{Example \ref{ex5} (3D-eddy-current)}
\begin{center}
\begin{tabular}{r|c|c|c}
$\Nelem$ & $\tnorm{(E,H)-(\Eh,\Hh)}$ & $\sqrts{\Mec(\Eh,\Hh)}$ & difference $\delta$ \\
\hline\hline
384    & 0.7228185218 & 0.7228185218 & 3.330669074e-16 \\
3072   & 0.3717887807 & 0.3717887807 & 6.106226635e-16 \\
24576  & 0.1883612515 & 0.1883612515 & 2.775557562e-16 \\
196608 & 0.0945757836 & 0.0945757836 & 8.604228441e-16
\end{tabular}
\end{center}
\label{tbl:Ex5}
\end{table}

\begin{ex}
\label{ex6}
We take the problem data of Example \ref{ex4} and solve the primal and dual problems 
in adaptively refined meshes with linear N\'ed\'elec elements and linear Courant elements, respectively. 
This finite element approximation pair will be denoted by $(\Eh,\Hh)$ and the linear systems are solved directly. 
We compare optimal refinement achieved by using the exact error distribution $e_T$ 
to the refinement provided by the distribution of the majorant $\eta_T$, where 
\begin{align*}
e_T^2
&:=\tnorm{(E,H)-(\Eh,\Hh)}^2_T
:=\norm{E-\Eh}_{\r(T)}^2
+\norm{H-\Hh}_{\ho(T)}^2,\\
\eta_T^2
&:=\Mec(\Eh,\Hh)_{T}
:=\norm{J-\Eh-\na^{\bot}\Hh}_{\lt(T)}^2
+\norm{\Hh-\rot\Eh}_{\lt(T)}^2
\end{align*}
and $T$ denotes an element (triangle) of the mesh discretization. 
We start from a regular mesh with $200$ elements, and perform nine refinement iterations, 
where on each iteration $30\%$ of elements with the highest amount of error are refined. 
The refinement of element meshes is done by regular refinement such that
the resulting mesh does not contain hanging nodes. 
The results of Figure \ref{fig:ex6} show that even though the equality is \emph{global}, 
the majorant can still be used to perform reliable adaptive computations. 
We see from Table \ref{tbl:Ex6} that the number of elements in the optimal meshes and
the meshes produced using $\eta_T$ are very close to each other.
In Figure \ref{fig:ex6mesh} we have depicted the meshes after the fourth refinement.
Figure \ref{fig:ex6mesh_final} depicts one of the finest parts of the final meshes.
In fact, the adaptive refinement using $\eta_T$ is very close to optimal in each step, 
and the resulting approximation after the last refinement is practically the same.
\end{ex}

\begin{figure}[!h]
\centering
\includegraphics[scale=0.55]{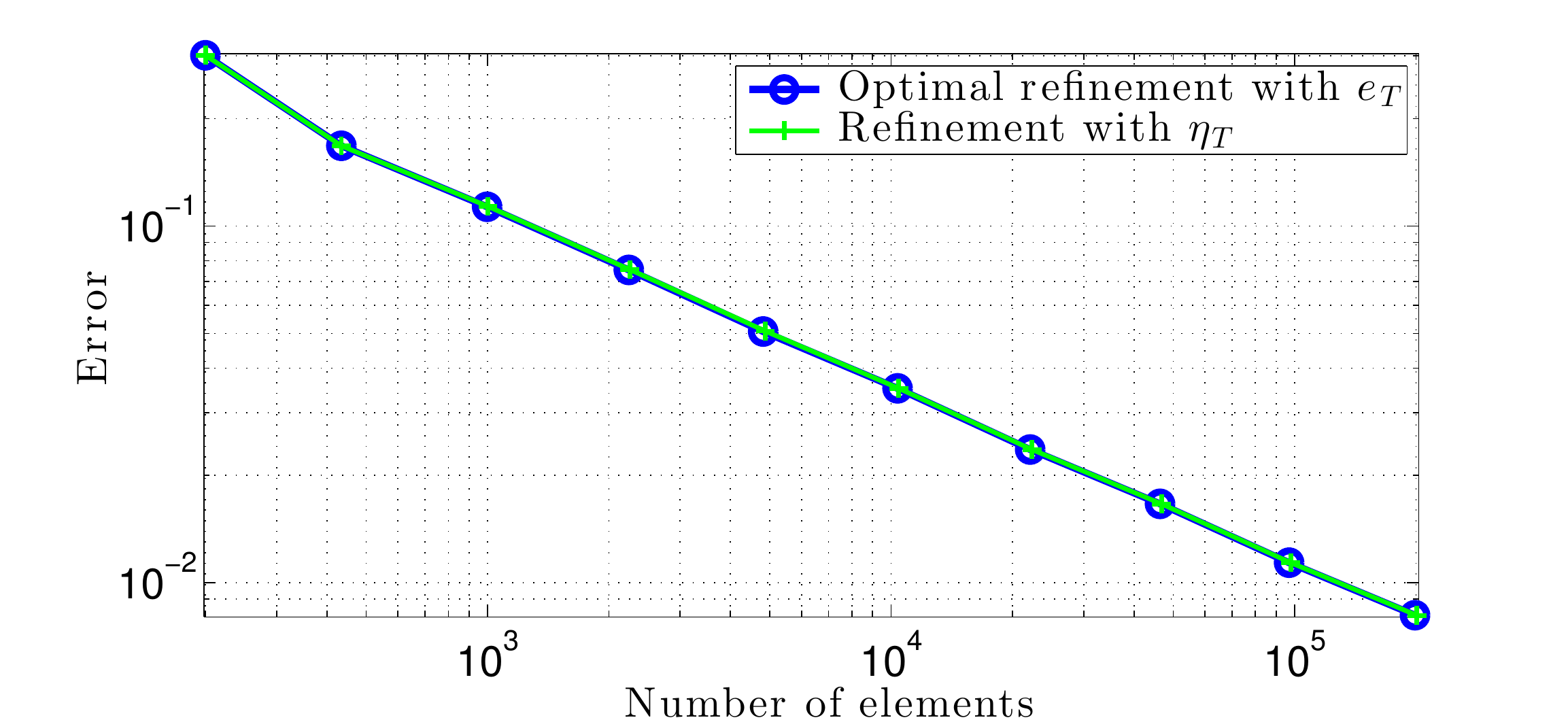}
\caption{Adaptive computation of Example \ref{ex6}, where the error is measured in the combined norm.}
\label{fig:ex6}
\end{figure}

\begin{figure}[!h] 
\centering
\includegraphics[scale=0.5,trim=0 25 0 0]{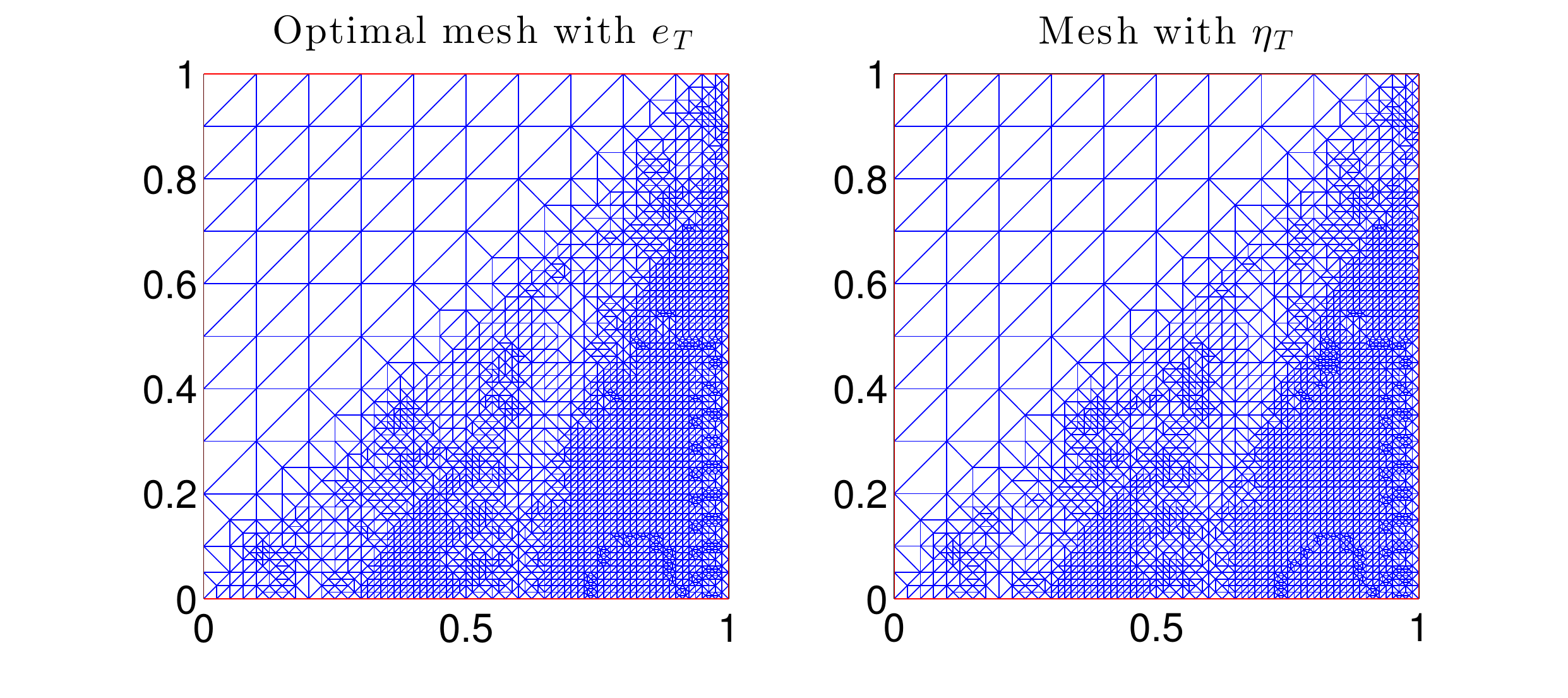}
\caption{Adaptive mesh after the fourth refinement in Example \ref{ex6}. 
There are 4823 elements in the optimal mesh, 
and 4878 elements in the mesh calculated with the help of $\eta_T$.}
\label{fig:ex6mesh}
\end{figure}

\begin{figure}[!h] 
\centering
\includegraphics[scale=0.5,trim=0 25 0 0]{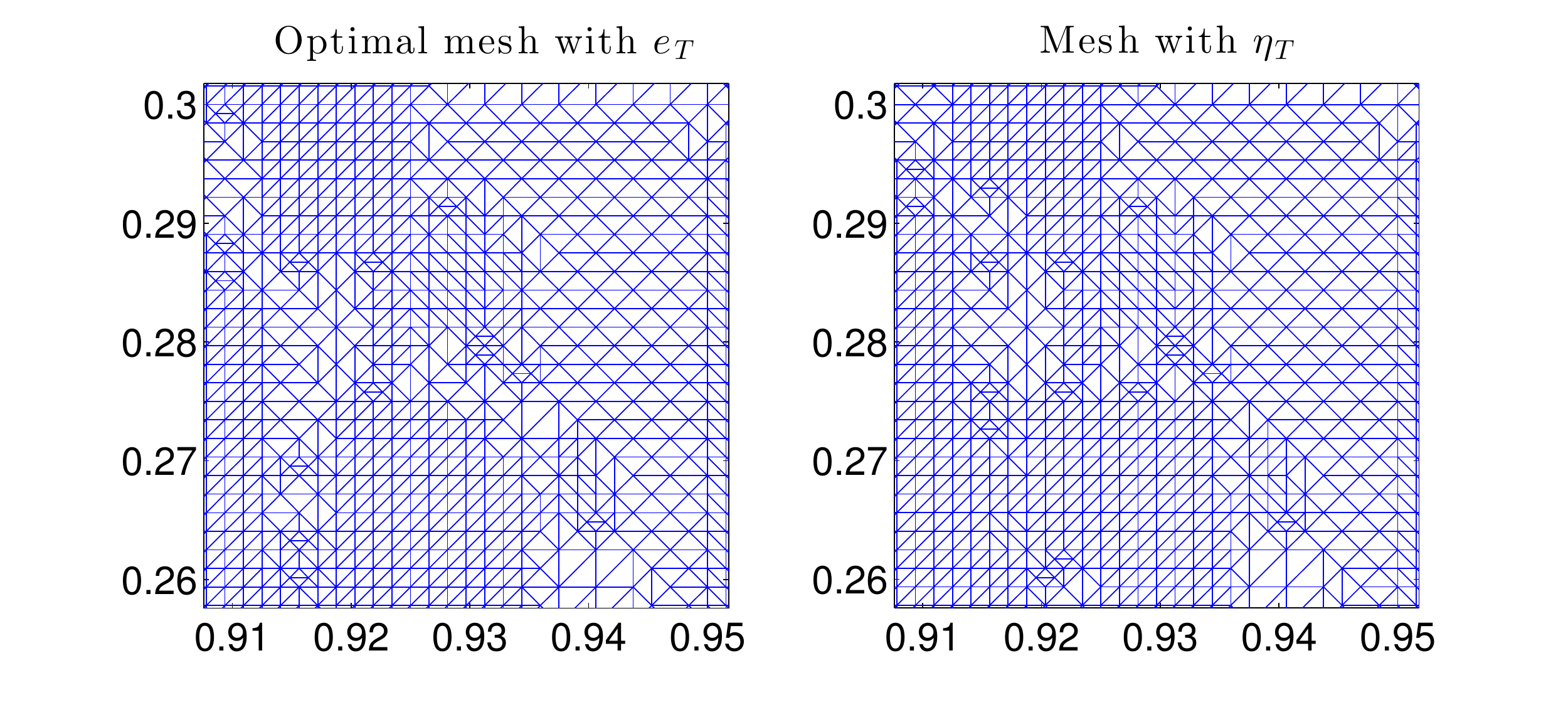}
\caption{One of the most fine parts in the final adaptive mesh in Example \ref{ex6}.}
\label{fig:ex6mesh_final}
\end{figure}

\begin{table} [htdp]
\caption{Adaptive computation of Example \ref{ex6}. 
The number of elements in the optimal meshes and the meshes generated by the help of $\eta_T$.}
\begin{center}
\begin{tabular}{r|r|r|r|l}
Ref. & optimal & with $\eta_T$ & difference & difference $\%$ \\
\hline\hline
- &    200 &    200 &    0 & 0 \\
1 &    434 &    434 &    0 & 0 \\
2 &    998 &   1002 &    4 & 0.40 \\
3 &   2240 &   2252 &   12 & 0.53 \\
4 &   4823 &   4878 &   55 & 1.14 \\
5 &  10378 &  10446 &  68  & 0.65 \\
6 &  22116 &  22337 &  221 & 0.99 \\
7 &  46388 &  46768 &  380 & 0.81 \\
8 &  96859 &  97832 &  973 & 1.00 \\
9 & 198704 & 200970 & 2266 & 1.14
\end{tabular}
\end{center}
\label{tbl:Ex6}
\end{table}

\begin{ex}
\label{ex7}
We take the 2D-eddy-current problem of Section \ref{subsec:EC2D} in the $L$-shaped domain 
$\om:=(0,1)^2\setminus\big([\nicefrac{1}{2},1]\times[0,\nicefrac{1}{2}]\big)$ 
with $\eps=\id$, $\mu=1000$ and $J=[1,0]^\top$. 
We set zero Dirichlet boundary conditions on the whole boundary, i.e., $\gad=\ga$ and $\gan=\emptyset$. 
The exact solution of this problem is unknown. 
However, since the majorant gives indeed the exact error in the combined norm, 
we will use this information in this example. Therefore, all the error values 
in Figure \ref{fig:ex7} and Table \ref{tbl:Ex7} are actually the values of the majorant. 
We compare uniform refinement and adaptive refinement using $\eta_T$ with
$$\eta_T^2
=\Mec(\Eh,\Hh)_{T}
=\norm{J-\Eh-\na^{\bot}\Hh}_{\lt(T)}^2
+\norm{\Hh-\mu^{-1}\rot\Eh}_{\lt(T),\mu}^2,$$
refining $30\%$ of elements on each refinement iteration as before. 
We solve the primal and dual problems with linear N\'ed\'elec elements and linear Courant elements, respectively.
The resulting linear systems are solved directly.
We see from Figure \ref{fig:ex7} that the adaptive procedure is beneficial in this example.
We have also depicted the approximation in Figure \ref{fig:ex7approx} 
and the mesh in Figure \ref{fig:ex7mesh} after the fifth refinement.
It can be concluded that in addition to providing the exact error, 
the majorant also provides a good error indicator without any additional computational expenditures.
\end{ex}

\begin{figure}[!h]
\centering
\includegraphics[scale=0.55]{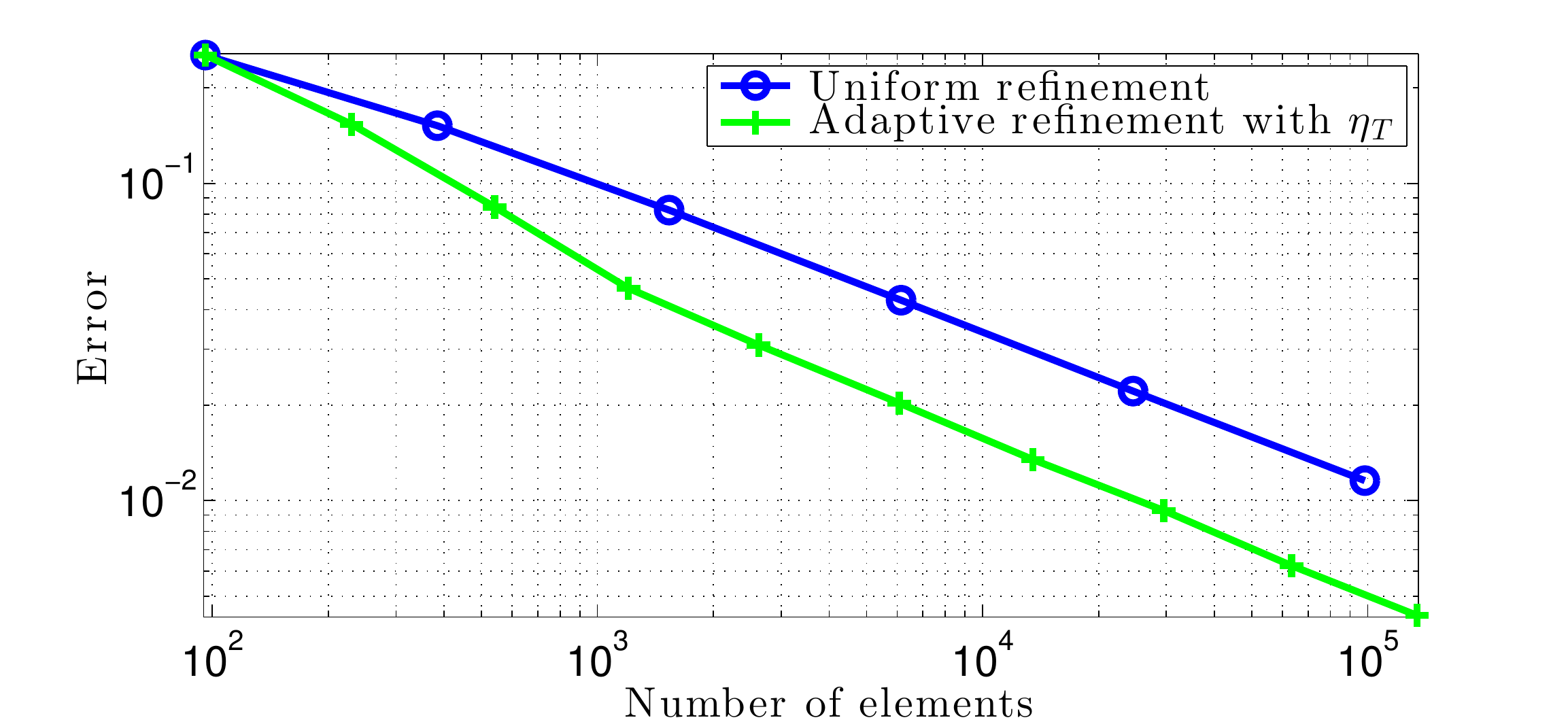}
\caption{Adaptive computation of Example \ref{ex7}.}
\label{fig:ex7}
\end{figure}

\begin{table} [htdp]
\caption{Example \ref{ex7} (2D-eddy-current) 
Adaptively refined meshes.}
\begin{center}
\begin{tabular}{r|c|c}
$\Nelem$ & $\sqrts{\Mec(\Eh,\Hh)}$ & $\sqrts{\Mec(\Eh,\Hh)}/\norm{J}_{\lt}$ \\
\hline\hline
96     & 0.2534 & 0.2926 \\
230    & 0.1534 & 0.1771 \\
541    & 0.0842 & 0.0973 \\
1204   & 0.0467 & 0.0539 \\
2623   & 0.0309 & 0.0357 \\
6082   & 0.0203 & 0.0234 \\
13514  & 0.0135 & 0.0155 \\
29530  & 0.0093 & 0.0107 \\
63363  & 0.0062 & 0.0072 \\
134205 & 0.0043 & 0.0050
\end{tabular}
\end{center}
\label{tbl:Ex7}
\end{table}

\begin{figure}[!h]
\centering
\begin{picture}(370,135)
\put(0,0){\includegraphics[scale=0.6,trim=0 0 0 0]{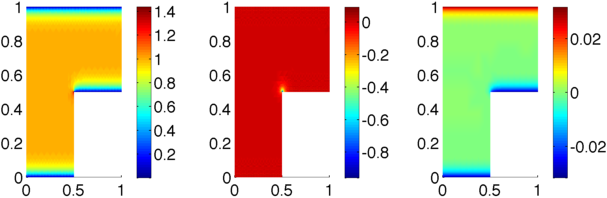}}
\put(35,125){\small$\Eho$}
\put(160,125){\small$\Eht$}
\put(290,125){\small$\Hh$}
\end{picture}
\caption{The two components of the approximate primal variable $\Eh$ and the dual variable $\Hh$ 
of Example \ref{ex7} after the third adaptive refinement.}
\label{fig:ex7approx}
\end{figure}

\begin{figure}[!h] 
\centering
\includegraphics[scale=0.4,trim=0 25 0 0]{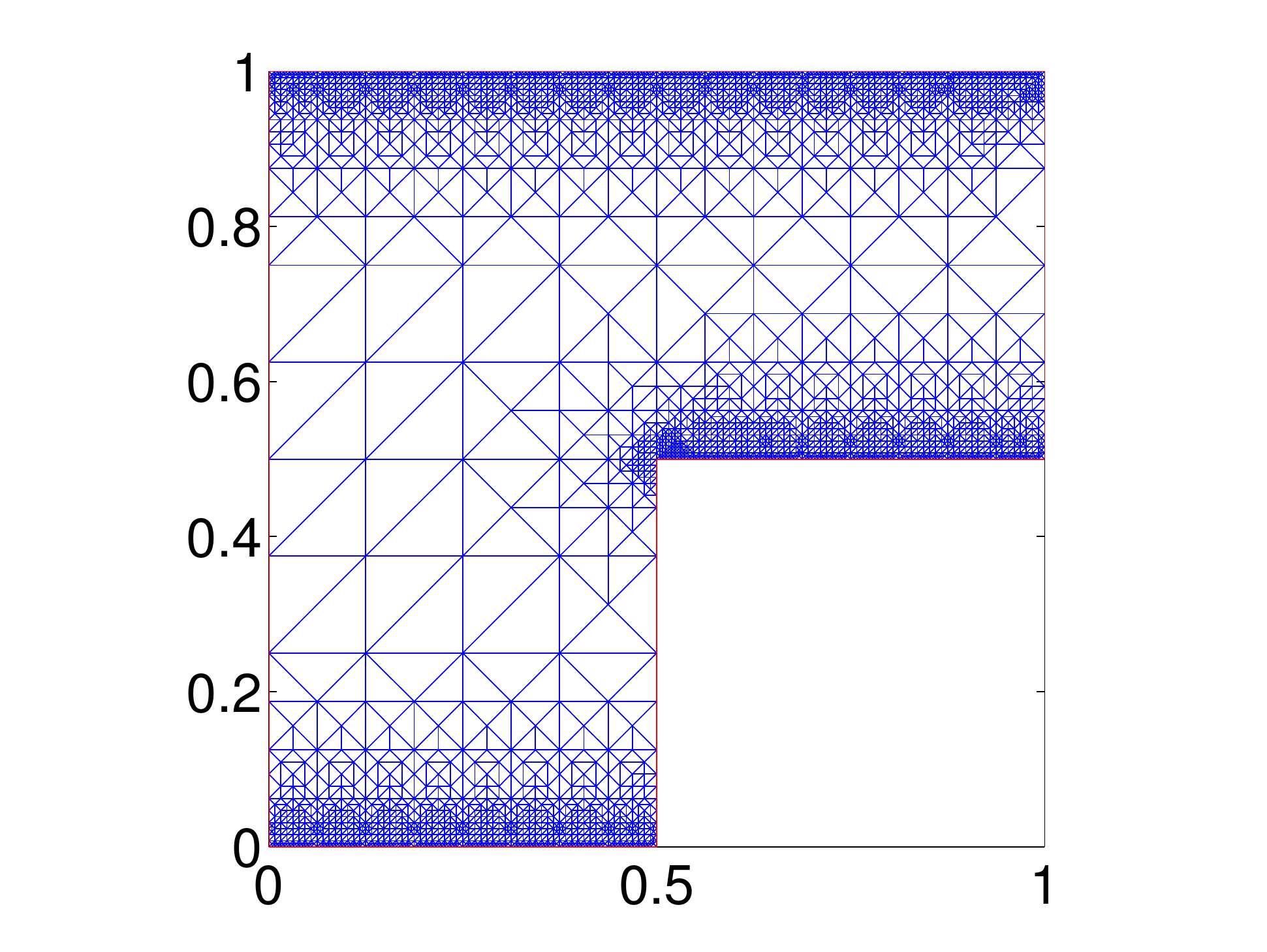}
\caption{Adaptive mesh after the fifth adaptive refinement in Example \ref{ex7}.}
\label{fig:ex7mesh}
\end{figure}

\begin{ex}
\label{ex8}
We take the 2D-eddy-current problem of Section \ref{subsec:EC2D} in $\om:=(0,1)^2$.
In order to define discontinuous data, we define with $\xi(x):=\ln(2+x_2)$ and
\begin{align*}
\om_1
&:=\big((0,1)\times(0.4,0.6)\big)\cup\big((0.3,0.5)\times(0,1)\big),&
\eps|_{\om_1}
&:=\id,&
\eps|_{\om\setminus\ol{\om}_1}
&:=100\cdot\id,\\
& &
\mu|_{\om_1}
&:=1000,&
\mu|_{\om\setminus\ol{\om}_1}
&:=1,\\
\om_2
&:=(0,1)\times(0.35,0.65),&
J|_{\om_2}
&:=\xi\begin{bmatrix}1\\0\end{bmatrix},&
J|_{\om\setminus\ol{\om}_2}
&:=-\xi\begin{bmatrix}0\\1\end{bmatrix}.
\end{align*}
We set zero Dirichlet boundary conditions on the right side of the rectanglular boundary
and zero Neumann boundary condition on the remaining part, i.e., 
$\gad=\set{x\in\om}{x_1=1}$. 
As in Example \ref{ex7}, the exact solution of this problem is unknown, so the error values 
in Figure \ref{fig:ex8} and Table \ref{tbl:Ex8} are the values of the majorant. 
We compare uniform refinement and adaptive refinement using $\eta_T$ with
$$\eta_T^2
=\Mec(\Eh,\Hh)_{T}
=\norm{J-\eps\Eh-\na^{\bot}\Hh}_{\lt(T),\eps^{-1}}^2
+\norm{\Hh-\mu^{-1}\rot\Eh}_{\lt(T),\mu}^2,$$
refining $30\%$ of elements on each refinement iteration as before. 
We solve the primal and dual problems with linear N\'ed\'elec elements and linear Courant elements, respectively.
The resulting linear systems are solved directly.
Again, we see from Figure \ref{fig:ex8} that the adaptive procedure is beneficial in this example.
We have also depicted the approximation in Figure \ref{fig:ex8approx} 
and the mesh in Figure \ref{fig:ex8mesh} after the third refinement.
\end{ex}

\begin{figure}[!h]
\centering
\includegraphics[scale=0.55]{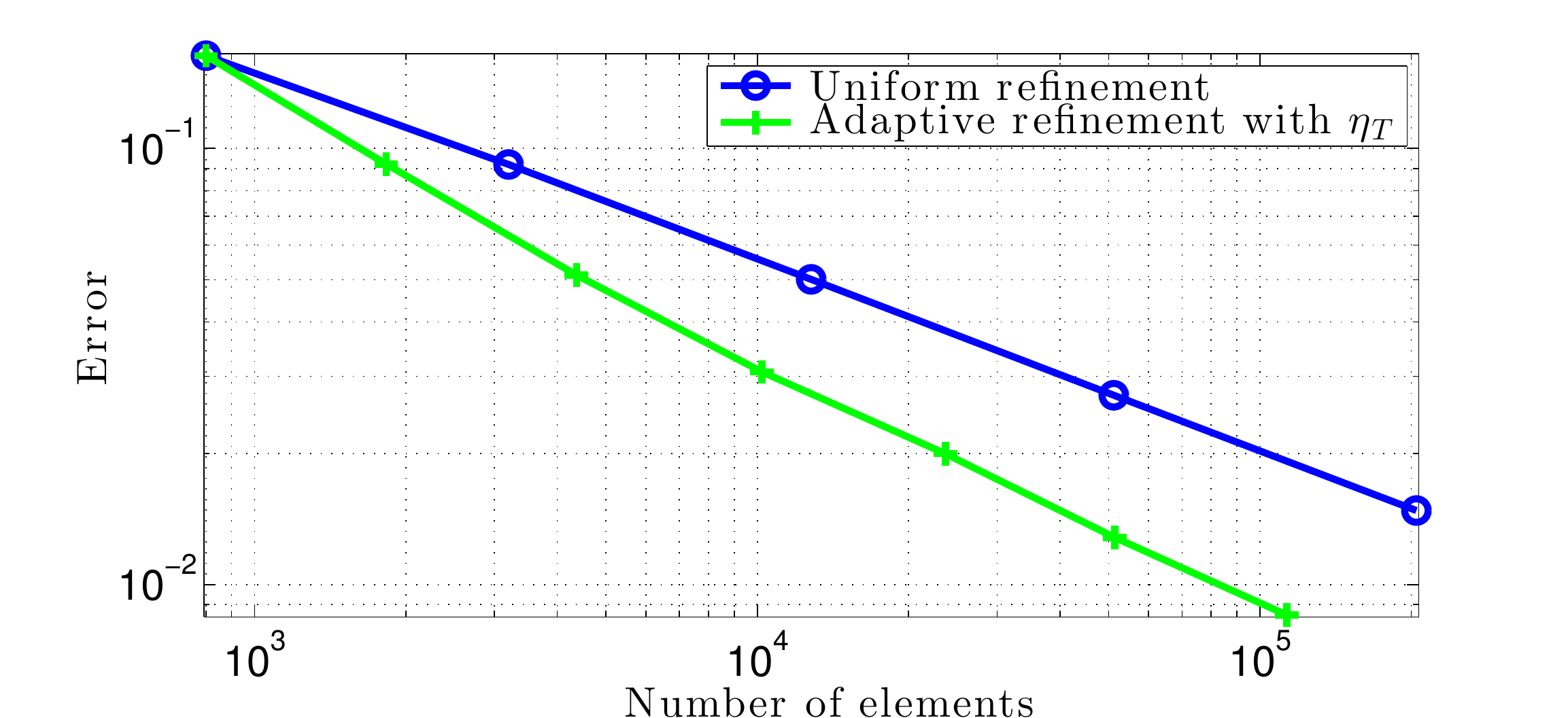}
\caption{Adaptive computation of Example \ref{ex8}.}
\label{fig:ex8}
\end{figure}

\begin{table} [htdp]
\caption{Example \ref{ex8} (2D-eddy-current) 
Adaptively refined meshes.}
\begin{center}
\begin{tabular}{r|c|c}
$\Nelem$ & $\sqrts{\Mec(\Eh,\Hh)}$ & $\sqrts{\Mec(\Eh,\Hh)}/\norm{J}_{\lt,\eps^{-1}}$ \\
\hline\hline
800    & 0.1632 & 0.2941 \\
1827   & 0.0921 & 0.1659 \\
4367   & 0.0513 & 0.0924 \\
10214  & 0.0307 & 0.0554 \\
23657  & 0.0199 & 0.0359 \\
51429  & 0.0128 & 0.0231 \\
113073 & 0.0085 & 0.0153
\end{tabular}
\end{center}
\label{tbl:Ex8}
\end{table}

\begin{figure}[!h]
\centering
\begin{picture}(370,135)
\put(0,0){\includegraphics[scale=0.6,trim=0 0 0 0]{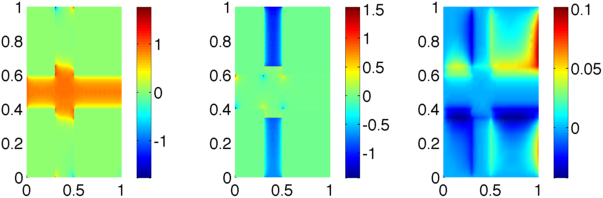}}
\put(35,125){\small$\Eho$}
\put(160,125){\small$\Eht$}
\put(290,125){\small$\Hh$}
\end{picture}
\caption{The two components of the approximate primal variable $\Eh$ and the dual variable $\Hh$ 
of Example \ref{ex8} after the third adaptive refinement.}
\label{fig:ex8approx}
\end{figure}

\begin{figure}[!h] 
\centering
\includegraphics[scale=0.4,trim=0 25 0 0]{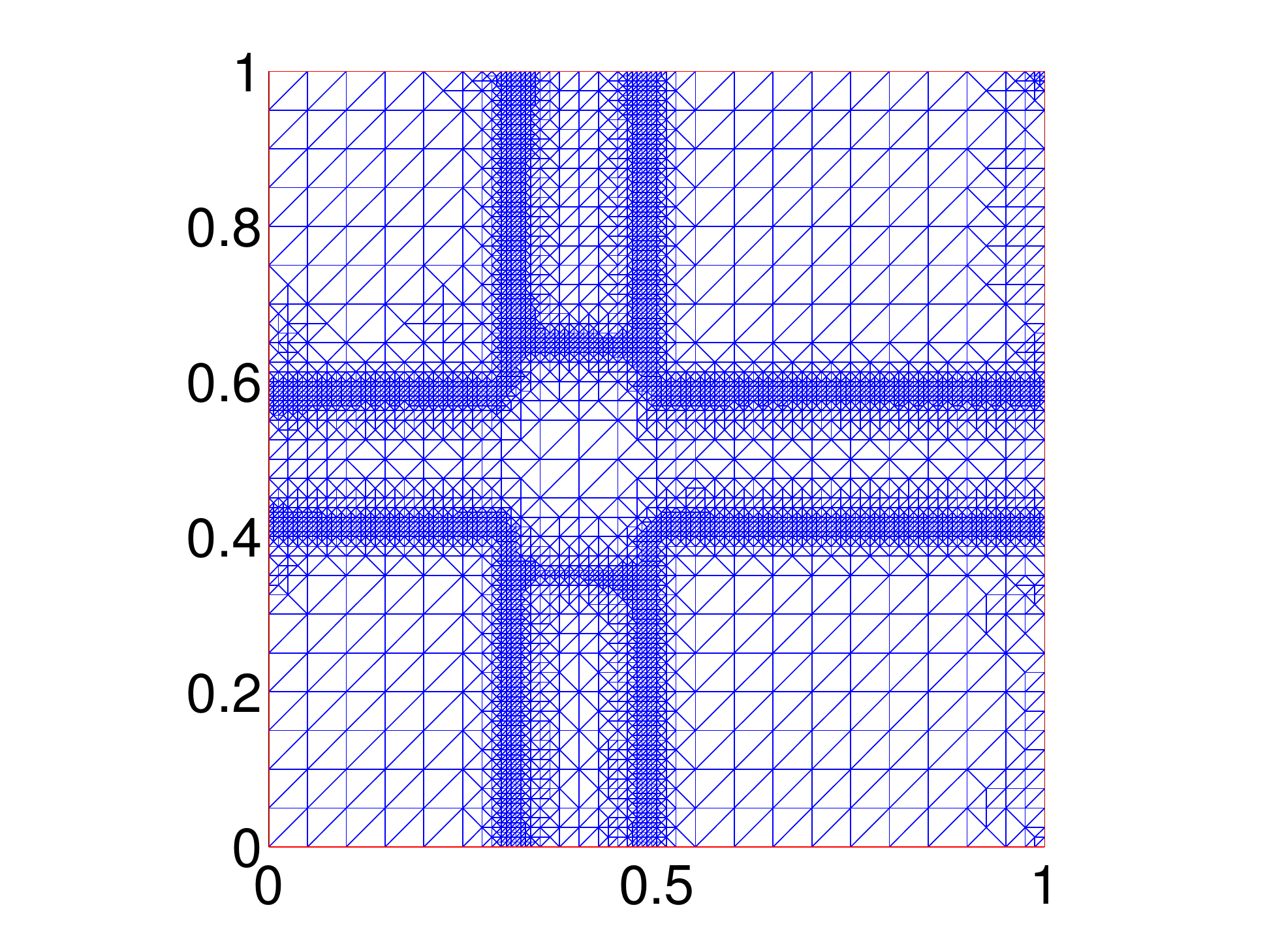}
\caption{Adaptive mesh after the third adaptive refinement in Example \ref{ex8}.}
\label{fig:ex8mesh}
\end{figure}

To conclude, in all the tests performed, nonzero values of $\delta$ were of magnitude $10^{-18}$-$10^{-15}$. 
This is within the limit of machine precision, 
so numerically these numbers are considered zero. 
In addition to verifying the equality, 
we also performed three simple examples to show that the majorant 
can be used to perform refinement of element meshes
without any additional computational expenditures.


\begin{acknow}{10}
We are deeply indebted to Sergey Repin for so many interesting
and encouraging discussions. We thank him gratefully not only 
for being an academic colleague and teacher but also a good friend.

The first author also thanks Jan Valdman for showing how to program vectorized FEM solvers.
The first author is funded by the Finnish foundations KAUTE Foundation and
V\"ais\"al\"a Foundation of the Finnish Academy of Science and Letters.

This contribution has been worked out mainly while the first author 
was visiting the Fakult\"at f\"ur Mathematik of the Universit\"at Duisburg-Essen during 2013.
\end{acknow}


\bibliographystyle{plain} 
\bibliography{/Users/paule/Library/texmf/tex/TeXinput/bibtex/paule}


\end{document}